\documentclass[12pt]{amsart}
\usepackage[usenames,dvipsnames]{color}
\usepackage{fancyvrb}
\usepackage[T1]{fontenc}
\usepackage{lmodern}
\usepackage{upquote,textcomp}
\usepackage{etoolbox}
\usepackage[12pt]{moresize}

\usepackage[latin1]{inputenc}
\usepackage[english]{babel}
\usepackage{amsmath}
\usepackage{amsthm}
\usepackage{amssymb}
\usepackage{amsmath, amsthm, amsfonts, mathrsfs, amsfonts}
\usepackage{latexsym}
\usepackage{textcomp}
\usepackage{enumerate}

\usepackage{graphicx}
\usepackage{tikz}
\usepackage{float}
\usepackage{wrapfig}
\usepackage{array}

\usepackage{amsaddr}
\usepackage{enumitem}

\theoremstyle{definition}
\newtheorem{definition}{Definition}[section]

\newtheorem{hp}[definition]{Hypothesis}

\theoremstyle{remark}
\newtheorem{remark}[definition]{Remark}

\theoremstyle{plain}

\theoremstyle{plain}
\newtheorem{cor}[definition]{Corollary}

\theoremstyle{plain}
\newtheorem{lemma}[definition]{Lemma}

\theoremstyle{plain}
\newtheorem{theorem}[definition]{Theorem}
\newtheorem{theoremA}[]{Theorem}

\theoremstyle{remark}

\newtheorem{notation}[definition]{Notation}

\theoremstyle{definition}

\newcommand{\A}{\mathbf{A}}

\newcommand{\C}{\mathrm{C}}

\newcommand{\F}{\mathcal{F}}
\newcommand{\G}{\mathrm{G}}

\newcommand{\N}{\mathrm{N}}

\newcommand{\J}{\mathrm{J}}

\newcommand{\Aut}{\mathrm{Aut}}
\newcommand{\Hom}{\mathrm{Hom}}

\newcommand{\Out}{\mathrm{Out}}
\newcommand{\Inn}{\mathrm{Inn}}
\newcommand{\Syl}{\mathrm{Syl}}

\newcommand{\coreF}{\mathrm{core_\F}}

\newcommand{\GF}{\mathrm{GF}}
\newcommand{\GL}{\mathrm{GL}}
\newcommand{\Sp}{\mathrm{Sp}}
\newcommand{\SL}{\mathrm{SL}}

\newcommand{\PGammaL}{\mathrm{P}\Gamma\mathrm{L}}
\newcommand{\PSL}{\mathrm{PSL}}

\newcommand{\PSp}{\mathrm{PSp}}

\newcommand{\norm}{\mathrel{\unlhd}}

\def \Z {\mathrm {Z}}

\def \Syl {\hbox {\rm Syl}}

\def \ov {\overline}

\title{Fusion systems on $p$-groups of sectional rank $3$}
\author{Valentina Grazian}
\address{Institute of Mathematics, University of Aberdeen, Fraser Noble Building, Aberdeen AB24 3UE, U.K.}
\email{valentinagrazian@libero.it}
\thanks{The author was supported by EPSRC}

\begin{document}

\maketitle

\begin{abstract} We study saturated fusion systems on $p$-groups having sectional rank $3$ for all odd primes $p$. For $p\geq 5$, we obtain a complete classification of the ones that do not have any non-trivial normal $p$-subgroups. 
\end{abstract}

\section*{Introduction}
The theory of fusion systems is a modern subject with applications in various branches of algebra. A fusion system $\F$ on a finite $p$-group $S$ is a category whose objects are the subgroups of $S$ and whose morphism sets $\Hom_\F(P,Q)$ between subgroups $P$ and $Q$ of $S$ are collections of injective morphisms that satisfy some axioms, first introduced by Puig (\cite{Pg}) and inspired by the conjugation action of a finite group on its $p$-subgroups.
Given a finite group $G$, there is a natural construction of a fusion system on one of its Sylow $p$-subgroups $S$: this is the category $\F_S(G)$ whose objects are the subgroups of $S$ and whose morphism sets are $\Hom_{\F_S(G)}(P,Q) = \{ c_g|_P \colon P \rightarrow Q \mid g\in G \text{ and } P^g\leq Q \}$, for every $P,Q\leq S$.
Many researchers around the world are currently working on classifying simple fusion systems at the prime $2$ and on classifying important families of simple fusion systems at odd primes. 
Taking inspiration from the Classification of Finite Simple Groups, an important class to examine is the class of $p$-groups of \emph{small sectional rank.} The rank of a finite group is the smallest size of a generating set for it and a $p$-group $S$ has sectional rank $k$ if every elementary abelian section  $Q/R$ of $S$ has order at most $p^k$ and $k$ is the smallest integer with this property (or equivalently if every subgroup of $S$ has rank at most $k$ and $k$ is the smallest integer with this property).
In the elementary case in which $S$ has sectional rank $1$, the group $S$ is cyclic and all saturated fusion systems on $S$ are completely determined by the automorphism group of $S$; this can be proved by adapting Burnside's result for groups with abelian Sylow $p$-subgroup (\cite{Burnside}).
All reduced fusion systems on $2$-groups of sectional rank at most $4$ have been classified by Oliver (\cite{2rank4}).
If $p$ is an odd prime, then the saturated fusion systems on $p$-groups of sectional rank $2$ have been classified by Diaz, Ruiz and Viruel (\cite{DRV}) and Parker and Semeraro (\cite{PSrank2}).

As a natural continuation of these works, in this paper we study saturated fusion systems on $p$-groups of sectional rank $3$ when $p$ is an odd prime. In particular, we classify all such fusion systems $\F$ whenever $p\geq 5$ and $\F$ satisfies the extra condition that $O_p(\F)=1$.

Let $p$ be an odd prime, let $S$ be a $p$-group and let $\F$ be a saturated fusion system on $S$.
The Alperin-Goldschmidt Fusion Theorem  \cite[Theorem 1.19]{AO} guarantees that $\F$ is completely determined by the $\F$-automorphisms of $S$ and by the $\F$-automorphisms of certain subgroups of $S$, called for this reason $\F$-essential subgroups of $S$ (Definition \ref{def.essential}).
If $Q \leq P$ are subgroups of $S$, we say that $Q$ is $\F$-characteristic in $P$ if $Q$ is normalized by $\Aut_\F(P)$. One of the axioms in the definition of a fusion system states that all the restrictions of conjugation maps realized by elements of $S$ belong to the fusion system. Hence if $Q$ is $\F$-characteristic in $P$ then $Q\norm \N_S(P)$.

If the $p$-group $S$ has sectional rank $3$ then by definition its subgroups have rank at most $3$. Since $\F$-essential subgroups are not cyclic, we start characterizing the ones that have rank $2$.
In \cite{pearls} we called $\F$-\emph{pearls} the $\F$-essential subgroups that are either elementary abelian of order $p^2$ or non-abelian of order $p^3$. Note that $\F$-pearls have rank $2$.

\begin{theoremA}\label{rank2pearl}
Suppose $p$ is an odd prime,  $S$ is a $p$-group of sectional rank $3$ and $\F$ is a saturated fusion system on $S$. Then every $\F$-essential subgroup $E$ of $S$ of rank $2$ that is not $\F$-characteristic in $S$ is an $\F$-pearl.
\end{theoremA}

If $O_p(\F)=1$ and there exists an $\F$-essential subgroup $E$ of $S$ that is $\F$-characteristic in $S$, then there must exist an $\F$-essential subgroup $P$ of $S$ distinct from $E$ (otherwise $E\leq O_p(\F)$, contradicting the assumptions).
With this in mind, we first undertake a deep study of the structure of $\F$-essential subgroups of rank $3$ that are not $\F$-characteristic in $S$ (Section 3). Then we study the interplay between distinct $\F$-essential subgroups that are $\F$-characteristic in $S$ (Section 4), using the classification of Weak BN-pairs of rank 2 presented in \cite{DGS}. This leads us to prove the following result.

\begin{theoremA}\label{normal}
Suppose $p$ is an odd prime,  $S$ is a $p$-group of sectional rank $3$ and $\F$ is a saturated fusion system on $S$ such that $O_p(\F)=1$. Then either 
\begin{itemize}
\item $S$ is isomorphic to a Sylow $p$-subgroup of the group $\Sp_4(p)$, or 
\item there exists an $\F$-essential subgroup of $S$ that is not normal in $S$ and there is at most one $\F$-essential subgroup of $S$ that is $\F$-characteristic in $S$.
\end{itemize}
\end{theoremA}

Note that if $S$ is a Sylow $p$-subgroup of the group $\Sp_4(p)$, then $S$ contains a maximal subgroup that is elementary abelian and the reduced fusion systems on $S$ are among the ones classified in \cite{p.index} and \cite{p.index2}.

Up to this point, our results hold for every odd prime $p$. The crucial distinction between $p=3$ and the other odd primes occurs when, given an $\F$-essential subgroup $E$ of $S$ of rank $3$, we look for a bound for the index of $E$ in $S$. We show that if $p\geq 5$ then every $\F$-essential subgroup of $S$ of rank $3$ has index $p$ in $S$ (Theorem \ref{pgeq5}), and so it is normal in $S$.
Combining this result with Theorems  \ref{rank2pearl} and \ref{normal} and using the classification of fusion systems containing pearls  given in \cite[Theorem B]{pearls}, we get our main result.

\begin{theoremA}\label{main}
Let $p\geq5$ be a prime, let $S$ be a $p$-group of sectional rank $3$ and let $\F$ be a saturated fusion system on $S$ such that $O_p(\F)=1$. Then $\F$ contains an $\F$-pearl and exactly one of the following holds:
\begin{enumerate}
\item $S$ is isomorphic to a Sylow $p$-subgroup of the group $\Sp_4(p)$;
\item $p=7$,
$S$ has order $7^5$, $S$ is uniquely determined up to isomorphism and \begin{itemize}
\item there exists a unique $\F$-conjugacy class  $E^\F$ of $\F$-essential subgroups of $S$, where $E\cong \C_7 \times \C_7$ and $\Out_\F(E)\cong \SL_2(7)$,
\item $\Out_\F(S) = \N_{\Out_\F(S)}(E)\cong \C_6$ and
\item $\F$ is unique up to isomorphism, simple and exotic.
\end{itemize}
\end{enumerate}
\end{theoremA}

More precisely, the group $S$ of order $7^5$ appearing in Theorem \ref{main} is the group stored in the software \emph{Magma} as \texttt{SmallGroup(7\string^5, 37)} and is isomorphic to a maximal subgroup of a Sylow $7$-subgroup of the Monster group. In fact, as shown in \cite{pearls}, the simple exotic fusion system $\F$ defined on $S$ is a subsystem of the $7$-fusion system of the Monster group.
We point out that the proof that $\F$ is exotic is the only part of the proof of Theorem \ref{main} that uses the Theorem of classification of finite simple groups.

If $p=3$, $S$ has sectional rank $3$  and there exists an $\F$-essential subgroup $E$ of $S$ of rank $2$, then by Theorem \ref{rank2pearl} the group $E$ is an $\F$-pearl and by  \cite[Theorem B]{pearls} the $3$-group $S$ is isomorphic to a Sylow $3$-subgroup of the group $\Sp_4(3)$. Note that by Theorem \ref{normal} this is true when $O_3(\F)=1$ and all the $\F$-essential subgroups of $S$ are normal in $S$.
However this is not always the case.

As we mentioned already, a crucial step toward the proof of Theorem \ref{main} is the fact that if $p\geq 5$ then every $\F$-essential subgroup of $S$ having rank $3$ is normal in $S$ (Theorem \ref{pgeq5}). This is not true for $p=3$.
For example if  $q \equiv 1 \mod 3$ and $S$ is a Sylow $3$-subgroup of $\SL_4(q)$, then $S$ has sectional rank $3$, $S\cong \C_{3^a} \wr \C_3$, where $3^a$ is the largest power of $3$ dividing $q-1$, and there is an $\F_S(\SL_4(q))$-essential subgroup $E$ of $S$ isomorphic to the central product $\C_{3^a} \circ 3^{1+2}_+$. In particular if $a\geq 2$ then $E$ is not normal in $S$. Such $3$-group $S$ has a maximal subgroup that is abelian and so the reduced fusion systems on it are among the ones classified in \cite{p.index, p.index2, p.index3}. However there are saturated fusion systems on $3$-groups of sectional rank $3$ in which every maximal subgroup is non-abelian. Examples are given by the $3$-fusion systems of the groups $\PGammaL_3(q^{3^a})$ for $q \equiv 1 \mod 3$ and $a\geq 1$.

\vspace{2mm}
\emph{Organization of the paper.} In Section 1 we study properties of $\F$-essential subgroups.
We characterize the $\F$-automorphism group of an $\F$-essential subgroups $E$ of $S$ of rank at most $3$,  showing that in many cases the group $O^{p'}(\Out_\F(E))$ is isomorphic to the group $\SL_2(p)$. We give sufficient conditions for an $\F$-essential subgroup of rank $2$ to be an $\F$-pearl (Theorem \ref{incenterfrat.is.pearl}), we introduce the concept of the normalizer tower of a subgroup (Definition   \ref{norm.tower}) and we show that abelian $\F$-essential subgroups of rank at most $3$ are not properly contained in any $\F$-essential subgroup of $S$ (Corollary \ref{lift.Ni.ab}).

In Section 2 we introduce the concept of the $\F$-core of a pair of $\F$-essential subgroups $E_1$ and $E_2$ of $S$ having the same normalizer $N$ in $S$ and we prove Theorem \ref{rank2pearl}.

Section 3 focuses on the structure of the $\F$-essential subgroups of a $p$-group $S$ that are not $\F$-characteristic in $S$ and whose normalizer in $S$ has sectional rank at most $3$. 

In Section 4 we suppose that the $p$-group $S$ has sectional rank $3$ and contains distinct $\F$-essential subgroups $E_1$ and $E_2$ both $\F$-characteristic in $S$. We prove that we can build a Weak BN-pair associated to $E_1$ and $E_2$ and  using the classification of Weak BN-pairs of rank $2$ contained in \cite{DGS} we show that if $O_p(\F)=1$ then $S$ is isomorphic to a Sylow $p$-subgroup of the group $\Sp_4(p)$  (Theorem \ref{Sp4}).

In Section 5 we prove Theorem \ref{normal}. To do that we study the $\F$-essential subgroups of $S$ whose automorphism group does not normalize the group $\Omega_1(\Z(S))$.

Finally in Section 6 we prove Theorem \ref{main}.

\vspace{2mm}
Throughout this paper $p$ is an odd prime, $S$ is a $p$-group and $\F$ is a saturated fusion system on $S$.

\section{Properties of $\F$-essential subgroups of small rank}
We refer to \cite[Chapter 1]{AO} for definitions and notations regarding the theory of fusion systems.
We recall here the definition of $\F$-essential subgroup.

\begin{definition}\label{def.essential}
A subgroup $E$ of $S$ is $\F$-essential if  the followings hold:
\begin{itemize}
\item $E$ is $\F$-centric: $\C_S(P) \leq P$ for every $P \in E^\F$;
\item $E$ is fully normalized in $\F$: $|\N_S(E)| \geq |\N_S(P)|$ for every $P\in E^\F$; and
\item $\Out_\F(E)$ contains a strongly $p$-embedded subgroup;
\end{itemize}
where $E^\F=\{ P \leq S \mid P = E\alpha \text{ for some } \alpha \in \Hom_\F(E,S) \}$ is the $\F$-conjugacy class of $E$ in $S$.
\end{definition}
Given an $\F$-essential subgroup $E$ of $S$, the normalizer fusion system $\N_\F(E)$ defined on $\N_S(E)$ is saturated (\cite[Theorem 1.11]{AO}) and constrained, and so it admits a model $G$ \cite[Theorem1.24]{AO}. In particular $G$ is a finite group such that $\N_S(E) \in \Syl_p(G)$, $E=O_p(G)$ and $G/E \cong \Out_\F(E)$.

The next two lemmas describe properties of $\F$-essential subgroups that we will use many times in this paper.
\begin{notation} If $P$ is a $p$-group then we write $\Phi(P)$ for the Frattini subgroup of $P$.
Recall that $\Phi(P) = [P,P]P^p$.
\end{notation}

\begin{definition}
Let $P$ be a $p$-group and let $\varphi \in \Aut(P)$. We say that $\varphi$ stabilizes the series of subgroups
$P_0 \leq P_1 \leq \dots \leq P_n$ of $P$ if for every $0 \leq i \leq n$ the morphism $\varphi$ normalizes $P_i$ and acts trivially on the quotient $P_i/P_{i-1}$ (for $i>0$).
\end{definition}

\begin{lemma}\label{char.series}
Let $E\leq S$ be a subgroup of $S$.
Consider the sequence of subgroups
\begin{equation}\label{series} E_0 \leq E_1 \leq \dots \leq E_n= E \end{equation}
such that $E_0\leq \Phi(E)$ and for every $0 \leq i \leq n$ the group $E_i$ is normalized by $\Aut_\F(E)$. If $\varphi \in \Aut_\F(E)$ stabilizes the series (\ref{series}) then $\varphi \in O_p(\Aut_\F(E))$.
\end{lemma}

\begin{proof}
By  \cite[Corollary 5.3.3]{GOR} the order of $\varphi$ is a power of $p$. Note that the set $H$ of all the morphisms in $\Aut_\F(E)$ stabilizing the series (\ref{series}) is a normal $p$-subgroup of $\Aut_\F(E)$. Hence $\varphi  \in H \leq O_p(\Aut_\F(E))$.
\end{proof}

\begin{lemma}\label{strict.frattini}
Let $E$ be an $\F$-essential subgroup of $S$. Then
\[ \Phi(E) < [\N_S(E),E]\Phi(E) < E.\]
\end{lemma}

\begin{proof}
If $[\N_S(E), E] \leq \Phi(E)$ then the automorphism group $\Aut_S(E)\cong \N_S(E)/\C_S(E)$ centralizes the quotient $E/\Phi(E)$. Hence $\Aut_S(E)$ is normal in $\Aut_\F(E)$, contradicting the fact that $E$ is $\F$-essential. So $\Phi(E) < [\N_S(E),E]\Phi(E)$. If $ [\N_S(E),E]\Phi(E) = E$ then $[\N_S(E),E]=E$, contradicting the fact that $S$ is nilpotent. Thus $[\N_S(E),E]\Phi(E) < E$.
\end{proof}

\begin{lemma}\label{GLr}
Let $E$ be an $\F$-essential subgroup of $S$. Then $\Out_\F(E)$ acts faithfully on $E/\Phi(E)$. In particular if $E/\Phi(E)$ has order $p^r$  then $\Out_\F(E)$ is isomorphic to a subgroup of $\GL_r(p)$.
\end{lemma}

\begin{proof}
By Lemma \ref{char.series} and the fact that $E$ is $\F$-essential we get $\C_{\Aut_\F(E)}(E/\Phi(E)) = \Inn(E)$. Hence  the group $\Out_\F(E)\cong \Aut_\F(E)/\Inn(E)$ acts faithfully on $E/\Phi(E)$. Since $E/\Phi(E)$ is elementary abelian (\cite[Theorem 5.1.3]{GOR}) we have $\Aut(E/\Phi(E)) \cong \GL_r(p)$, and we conclude.
\end{proof}

\begin{lemma}\label{indexp}
Let $E$ be an $\F$-essential subgroup of $S$. If $[\N_S(E) \colon E] =p$ then every subgroup of $S$ that is $\F$-conjugate to $E$ is $\F$-essential.
\end{lemma}

\begin{proof}
Let $P\leq S$ be a subgroup of $S$ that is $\F$-conjugate to $E$. Since $E$ is $\F$-centric, the group $P$ is $\F$-centric. By assumption $P=E\alpha$ for some $\alpha\in \Hom_\F(E,P)$. Hence $\Out_\F(E) \cong \Out_\F(P)$ via the map $\varphi \mapsto \alpha^{-1}\varphi \alpha$ and so the group $\Out_\F(P)$ has a strongly $p$-embedded subgroup. It remains to show that $P$ is fully normalized in $\F$. Since $E$ is fully normalized in $\F$ and $[\N_S(E) \colon E]=p$ we have
\[ |P|< |\N_S(P)| \leq |\N_S(E)| = [\N_S(E)\colon E]|E|=p|E|=p|P|.\]
Therefore $|\N_S(P)|=|\N_S(E)|$ and the group $P$ is fully normalized in $\F$. Hence $P$ is $\F$-essential.
\end{proof}

We now focus our attention on $\F$-essential subgroups of $S$ having rank at most $3$, since these will be the only ones that occur in a $p$-group of sectional rank $3$.   Note that we are not making any assumption on the sectional rank of $S$ yet.

\begin{theorem}\label{class.3}\cite[Theorem 1.7]{pearls}
Let $E\leq S$ be an $\F$-essential subgroup of $S$ of rank $r\leq 3$. Then $\Out_\F(E)$ is isomorphic to a subgroup of $\GL_r(p)$ and one of the following holds
\begin{itemize}
\item $r=2$ and $O^{p'}(\Out_\F(E))\cong \SL_2(p)$;
\item $r=3$, the action of $\Out_\F(E)$ on $E/\Phi(E)$ is reducible, $O^{p'}(\Out_\F(E))\cong \SL_2(p)$ and
$\Out_\F(E)$ is isomorphic to a subgroup of $\GL_2(p) \times \GL_1(p)$; or
\item $r=3$,  the action of $\Out_\F(E)$ on $E/\Phi(E)$ is irreducible and the group  $O^{p'}(\Out_\F(E))$ is isomorphic to one of the following groups:
\begin{enumerate}
\item $\SL_2(p)$;
\item $\PSL_2(p)$;
\item the Frobenius group $13: 3$ with $p=3$.
\end{enumerate}
\end{itemize}
In particular $[\N_S(E) \colon E]=p$ and every subgroup of $S$ that is $\F$-conjugate to $E$ is $\F$-essential.
\end{theorem}

We state Stellmacher's Pushing Up Theorem (\cite[Theorem 1]{stell}), that is used in the proof of Lemma \ref{frattini.is.char} and is crucial in the proof of Theorem \ref{characterization.E}. We present the complete statement, that includes also the case $p=2$ (whereas in the rest of this paper we recall that $p$ is an odd prime).

\begin{theorem}[Stellmacher's Pushing Up Theorem]\label{stell.here}
Let $G$ be a finite group, $p$ a prime and $P$ a Sylow $p$-subgroup of $G$ such that
\begin{enumerate}
\item No non-trivial characteristic subgroup of $P$ is normal in $G$, and
\item  $\ov{G}/\Phi(\ov{G}) \cong \PSL_2(p^n)$ for $\ov{G}=G/O_p(G)$.
\end{enumerate}
Let $Q=O_p(G)$ and $V=[Q, O^p(G)]$.
Then either $P$ is elementary abelian or there exists $\alpha \in \Aut(P)$ such that
\[ L/V_0O_{p'}(L) \cong \SL_2(p^n) \]
where $L= V^{\alpha}O^p(G)$ and $V_0 = V(L \cap \Z(G))$, and one of the following holds:
\begin{enumerate}
\item $P/\Omega_1(\Z(P))$ is elementary abelian, $V \leq \Z(Q)$ and $V$ is a natural $\SL_2(p^n)$-module for $L/V_0O_{p'}(L)$;
\item $p=2$, $P/\Omega_1(\Z(P))$ is elementary abelian, $V \leq \Z(Q)$, $n>1$ and $V/(V \cap \Z(G))$    is a natural $\SL_2(2^n)$-module for $L/V_0O_{2'}(L)$;
\item $p\neq 2$, $\Z(V) \leq \Z(Q)$,  $\Phi(V) = V \cap \Z(G)$ has order $p^n$, and $V/\Z(V)$ and
$\Z(V)/\Phi(V)$ are natural $\SL_2(p^n)$-modules for $L/V_0O_{p'}(L)$.

\end{enumerate}
In addition, in case $(3)$ the group $P$ has nilpotency class $3$, $\Phi(\Phi(P))=1$ and $P$ does not act quadratically on $V/\Phi(V)$.
\end{theorem}

\begin{lemma}\label{frattini.is.char}
Let $E$ be an $\F$-essential subgroup of $S$ of rank $2$. Then its Frattini subgroup $\Phi(E)$ is $\F$-characteristic in $\N_S(E)$.
\end{lemma}

\begin{proof} Set $\N= \N_S(E)$.
Since $E$ has rank $2$, by Theorem \ref{class.3} we have that $[\N \colon E]=p$ and $O^{p'}(\Out_\F(E))\cong \SL_2(p)$.
If $E$ is $\F$-characteristic in $\N$ then $\Phi(E)$ is $\F$-characteristic in $\N$.
Suppose that $E$ is not $\F$-characteristic in $\N$ and let $T$ be the largest subgroup of $E$ that is $\F$-characteristic in $E$ and $\N$. Then $T<E$ and since $E$ is $\F$-essential of rank $2$, by Lemma \ref{char.series} we get $T\leq \Phi(E)$.
Let $G$ be a model for $\N_\F(E)$. Then $E=O_p(G) \norm G$ and $G/E \cong \Out_\F(E)$. Note that $T \norm G$ because  $T$ is $\F$-characteristic in $E$. Set  $\ov{G}=G/T$. By Stellmacher's Pushing Up Theorem (Theorem \ref{stell.here}) applied to $\ov{O^{p'}(G)}$ and $\ov{\N} \in \Syl_p(\ov{O^{p'}(G)})$ we deduce that exactly one of the following holds:
\begin{enumerate}
\item the quotient $\ov{\N}/\Omega_1(\Z(\ov{\N}))$ is elementary abelian;
\item $\Phi(\Phi(\ov{\N})) = 1$ and there exists a subgroup $\ov{V}$ of $\ov{E}$ such that $\ov{V}/\Z(\ov{V})$ and $\Z(\ov{V})/\Phi(\ov{V})$ are natural $\SL_2(p)$-modules for $O^{p'}(G)/E \cong \SL_2(p)$ (in particular $\ov{V}$ has rank $4$).
\end{enumerate}

Suppose that we are in the second case. Then $\Phi(\ov{\N})=\ov{\Phi(\N)}$ is elementary abelian and since $[\ov{V} \colon \Z(\ov{V})]=p^2$, we deduce that $\ov{V} \nleq \ov{\Phi(\N)}$. Thus $\ov{E} = \ov{V}\Phi(\ov{E})$ and so $\ov{E}=\ov{V}$. However by assumption the group $\ov{E}$ has rank $2$, contradicting the fact that $\ov{V}$ has rank $4$.
Therefore the second case cannot occur and so the quotient $\ov{\N}/\Omega_1(\Z(\ov{\N}))$ is elementary abelian.
Note that $\Phi(E) < \Phi(\N) < E$ by Lemma \ref{strict.frattini} and the fact that $[\N \colon E]=p$.
Also, $\N/E$ does not centralize $E/\Phi(E)$. Thus $\Omega_1(\Z(\ov{\N})) \leq \ov{E}$. Hence $\Omega_1(\Z(\ov{\N}))=\ov{\Phi(\N)}$ and so the group $\ov{E}$ is abelian. Moreover $\Omega_1(\Z(\ov{\N})) < \Omega_1(\ov{E})$ by maximality of $T$. By assumption $[\ov{E} \colon \Phi(\ov{E})]=p^2$, so $[\ov{E} \colon \Omega_1(\Z(\ov{\N}))] = p$ and we conclude that $\ov{E} =\Omega_1(\ov{E})$. Since $\ov{E}$ is abelian, we deduce that it is elementary abelian. Therefore $T=\Phi(E)$ and the group $\Phi(E)$ is $\F$-characteristic in $\N$.
\end{proof}

\begin{lemma}\label{critical}
Let $E$ be an $\F$-essential subgroup of $S$ of rank $2$. Suppose there exists an automorphism $\varphi \in \Aut_\F(E)$ of order prime to $p$ that centralizes $\Phi(E)$. Then $\Phi(E)=[E,E]$ and $E$ has exponent $p$.
\end{lemma}

\begin{proof}
By \cite[Proposition 11.11]{GLS2} and the fact that $p$ is odd we deduce that there exists a characteristic subgroup $C$ of $E$ such that
\begin{enumerate}
\item $\varphi$ acts faithfully on $C$;
\item $\Phi(C)=[C,C]$ is elementary abelian;
\item either $C$ is abelian or $C$ has exponent $p$.
\end{enumerate}
By assumption $\varphi$ centralizes $\Phi(E)$, so $C\nleq \Phi(E)$. Since $E$ is $\F$-essential and has rank $2$, by Lemma \ref{char.series} we get $C\Phi(E)=E$ and so $E=C$. Thus $\Phi(E)=[E,E]$. Finally notice that if $E$ is abelian then $\Phi(E)=[E,E]=1$ and so in any case the group $E$ has exponent $p$.
\end{proof}

\begin{theorem}\label{incenterfrat.is.pearl}
Let $E$ be an $\F$-essential subgroup of $S$ of rank $2$. If $E$ is not $\F$-characteristic in $\N_S(E)$ and $\Phi(E) \leq \Z(E)$ then $E$ is an $\F$-pearl.
\end{theorem}

\begin{proof}
Note that $[E \colon \Z(E)] \leq [E \colon \Phi(E)]=p^2$ and so $|[E,E]|\leq p$.
Set $\N=\N_S(E)$.
By assumption $E=\C_E(\Phi(E)) \leq \C_\N(\Phi(E))$. By Lemma \ref{frattini.is.char} the group $\C_\N(\Phi(E))$ is $\F$-characteristic in $\N$. Since $[\N \colon E]=p$ by Theorem \ref{class.3} and $E$ is not $\F$-characteristic in $\N$, we deduce that $\N=\C_\N(\Phi(E))$. Thus $\Phi(E)$ is centralized by $O^{p'}(\Aut_\F(E)) = \langle \Aut_S(E) ^{\Aut_\F(E)} \rangle$. Therefore by Lemma \ref{critical} we get that $\Phi(E)=[E,E]$ and $E$ has exponent $p$.  In particular $|\Phi(E)|=|[E,E]| \leq p$ and so $E$ is an $\F$-pearl.
\end{proof}

\begin{lemma}\label{norm.rank3.SL2p}
Let $E$ be an $\F$-essential subgroup of $S$ of rank $3$.  If $\N_S(E)$ has rank $3$ then
\[ O^{p'}(\Out_\F(E))\cong  \SL_2(p).\]
\end{lemma}

\begin{proof}
Set $\N=\N_S(E)$.
By Lemma \ref{strict.frattini} we have $\Phi(E) < [E,\N]\Phi(E) \leq \Phi(\N)$ and by Theorem \ref{class.3} we have $[\N \colon E] = p$, which implies $\Phi(\N) \leq E$. Since $\N$ has rank $3$ we get $[E \colon \Phi(\N)]  = p^2$ and $[\Phi(\N) \colon \Phi(E)]  =p$. In particular
\[[E,\N]\Phi(E) = \Phi(\N)\quad \text{ and } \quad[E,\N,\N] \leq \Phi(E).\]
Hence the group $\Out_S(E)\cong \N/E$ acts quadratically on the elementary abelian $p$-group  $E/\Phi(E)$.
Also, $\Out_\F(E)$ acts faithfully on $E/\Phi(E)$ by Lemma \ref{GLr} and $O_p(\Out_\F(E))=1$ since $\Out_\F(E)$ has a strongly $p$-embedded subgroup.
Therefore by \cite[Theorem 3.8.3]{GOR} the group $\Out_\F(E)$ involves $\SL_2(p)$ and by Theorem \ref{class.3} we conclude $O^{p'}(\Out_\F(E)) \cong \SL_2(p)$.
\end{proof}

\begin{lemma}\label{quad.action}
Let $E$ be an $\F$-essential subgroup of $S$ of rank $3$, let $G$ be a model for $\N_\F(E)$ and let $T$ be a subgroup of $E$ that is $\F$-characteristic in $E$ and $\N_S(E)$.
Assume that there exist subgroups $H\leq \N_S(E) \backslash E$ and $V\leq E$ both containing $T$ such that  $V/T\leq \Omega_1(\Z(E/T))$, $V$ is normal in $G$ and $H/T$ acts quadratically on $V/T$. Then
\[O^{p'}(\Out_\F(E))\cong \SL_2(p).\]
\end{lemma}

\begin{proof}
Let $A=\langle (\N_S(E))^G\rangle$; so $A/E \cong O^{p'}(\Out_\F(E))$.
To simplify notation we set $\N=\N_S(E)$ and we may assume that $T=1$.
First notice that the group $V$ is elementary abelian.
By assumption $E$ has rank $3$ and so by Theorem \ref{class.3} the quotient $\N/E$ has order $p$. From $H\nleq E$ we get $\N=EH$.
Since $H$ acts quadratically on $V$, we have $[V,H]\neq 1$. In particular $H\nleq \Z(\N)$ and $E=\C_{\N}(V)$.
Thus
\[\N/E\cong H/(E \cap H)=H/\C_H(V) ~ \text{ and } ~ \N/E \cong \C_{A}(V)\N/\C_{A}(V).\]

Therefore $\N/E$ is isomorphic to a $p$-subgroup of $A/\C_{A}(V)$ that acts quadratically on $V$. Note that $A/\C_{A}(V)$ acts faithfully on $V$. Since $E$ is $\F$-essential, the group $G/E$ has a strongly $p$-embedded subgroup. Thus $O_p(A/E)=1$. Also, $\N\in \Syl_p(A)$ and so $O_p(A/\C_{A}(V))=1$.
By \cite[Theorem 3.8.3]{GOR} we deduce that $A/\C_{A}(V)$ involves $\SL_2(p)$. Hence $A/E\cong O^{p'}(\Out_\F(E))$ involves $\SL_2(p)$ and by Theorem \ref{class.3} we conclude that $O^{p'}(\Out_\F(E))\cong \SL_2(p)$.
\end{proof}

\begin{lemma}\label{thompson.notin}
Let $E$ be an $\F$-essential subgroup of $S$ of rank $3$ and let $T$ be a subgroup of $E$ that is $\F$-characteristic in $E$ and $\N_S(E)$. Set  $\ov{\N}=\N_S(E)/T$ and $\ov{E} = E/T$.
If $\J(\ov{\N}) \nleq \ov{E}$ and $\Omega_1(\Z(\ov{E})) \neq \Omega_1(\Z(\ov{\N}))$ then \[ O^{p'}(\Out_\F(E))\cong \SL_2(p),\]
where $\J(\ov{\N})$ is the Thompson subgroup of $\ov{\N}$, i.e. $\J(\ov{\N}) = \langle \A(\ov{\N})\rangle$, where $\A(\ov{\N})$ is the set of abelian subgroups of $\ov{\N}$ having order $\max \{ |A| \mid A\leq \ov{\N} \text{ and } A \text{ is abelian}\}$.

\end{lemma}

\begin{proof}
To simplify notation assume $T=1$. Set $\N=\N_S(E)$ and $V=\Omega_1(\Z(E))$. Let $H \in \A(\N)$ be such that $H\nleq E$ and $|V \cap H|$ is maximal.
By assumption and Theorem \ref{class.3} we have $[\N \colon E] = p$. So $\N=EH$ and since $V\neq \Omega_1(\Z(\N))$ and $H$ is abelian we deduce that $[V,H]\neq 1$. In particular $V\nleq H$.

Note that $V$ is normal in $\N$, so it is normalized by $H$. If $V$ normalizes $H$ then $[V,H,H] \leq [H,H]=1$. So $H$ acts quadratically on $V$ and by Lemma   \ref{quad.action} we conclude $O^{p'}(\Out_\F(E))\cong \SL_2(p)$.

Suppose for a contradiction that $V$ does not normalize $H$.
Then by the Thompson replacement theorem (\cite[Theorem 8.2.5]{GOR}) there exists an abelian subgroup $H^*\in \A(\N)$ such that $V\cap H < V \cap H^*$ and $H^*$ normalizes $H$. Since $|V\cap H|$ is maximal by the choice of $H$, we have $H^*\leq E$. Therefore $V\leq H^*$, by maximality of $|H^*|$,  and so $V$ normalizes $H$, a contradiction.
\end{proof}

\begin{theorem}\label{typeIIauto}
Let $E$ be an $\F$-essential subgroup of $S$ of rank at most $3$ that is not $\F$-characteristic in $S$.
Then $O^{p'}(\Out_\F(E))\cong \SL_2(p)$.
\end{theorem}

\begin{proof}
If $E$ has rank $2$ then the statement follows from Theorem \ref{class.3}. Suppose $E$ has rank $3$, set $\N=\N_S(E)$ and let $\alpha \in \Aut_\F(\N)$ be such that $E\neq E\alpha$. By Theorem \ref{class.3} the group $E\alpha$ is an $\F$-essential subgroup of $S$. Let $T$ be the largest subgroup of $E \cap E\alpha$ that is normalized by $\Aut_\F(E)$, $\Aut_\F(E\alpha)$ and $\Aut_\F(N)$ and set $\ov{\N}=\N/T$. If the Thompson subgroup $\J(\ov{\N})$ is contained in $\ov{E} \cap \ov{E}\alpha$ then $\J(\ov{\N})=\J(\ov{E}) = \J(\ov{E\alpha})$ and so $\J(\ov{\N}) =1$ by the maximality of $T$, which is a contradiction. Hence $\J(\ov{\N}) \nleq \ov{E} \cap \ov{E\alpha}$ and since $\J(\ov{\N})=\J(\ov{\N})\alpha$, we deduce that $\J(\ov{\N})\nleq \ov{E}$. Note that $\Omega_1(\Z(\ov{E})) \neq \Omega_1(\Z(\ov{\N}))$ by maximality of $T$. Therefore by  Lemma \ref{thompson.notin} we get that $O^{p'}(\Out_\F(E))\cong \SL_2(p)$.
\end{proof}

We now focus on the $\F$-essential subgroups that are $\F$-characteristic in $S$ and have \emph{sectional rank} at most $3$.

\begin{theorem}\label{SL2p} Let $E_1 \leq S$ and $E_2 \leq S$ be distinct $\F$-essential subgroups of $S$. Suppose that $E_1$ and $E_2$ are $\F$-characteristic in $S$ and have sectional rank at most $3$. Then there exists $i\in \{1,2\}$ such that
\[O^{p'}(\Out_\F(E_i))\cong \SL_2(p).\]
\end{theorem}

\begin{proof} By assumption and the fact that $\F$-essential subgroups are not cyclic, $E_i$ has rank either $2$ or $3$ for every $i$.
If $E_i$ has rank $2$ for some $i\in \{1,2\}$ then $O^{p'}(\Out_\F(E_i))\cong \SL_2(p)$ by Theorem \ref{class.3}. So we may assume that both $E_1$ and $E_2$ have rank $3$.
Let $G_i$ be a model for $\N_\F(E_i)$ and let $T$ be the largest subgroup of $E_1 \cap E_2$ that is normalized by $\Aut_\F(E_1)$, $\Aut_\F(E_2)$ and $\Aut_\F(S)$.
To simplify notation we assume $T=1$. Set
\[Z=\Omega_1(\Z(S)) \quad \text{ and } \quad V_i = \langle Z^{G_i} \rangle \leq \Omega_1(\Z(E_i)).\]
Let $\J(S)$ be the Thompson subgroup of $S$. If $\J(S) \leq E_1 \cap E_2$ then $\J(S)=\J(E_1)=\J(E_2) = 1$ by maximality of $T$, giving a contradiction. Thus we may assume $\J(S) \nleq E_1$.
Hence by Lemma \ref{thompson.notin} if $\Omega_1(\Z(E_1)) \neq Z$ then $O^{p'}(\Out_\F(E_1)) \cong \SL_2(p)$ and we are done.

Suppose $\Omega_1(\Z(E_1))=Z$.
By maximality of $T$ the group $Z$ is not $\F$-characteristic in $E_2$. In particular $Z < V_2$.
Note that $V_2$ is an elementary abelian subgroup of $E_2$, that has sectional rank $3$. Therefore either $|V_2|=p^2$ (and $|Z|=p$) or $|V_2|=p^3$.

Suppose $|V_2| = p^2$. Then $G_2/\C_{G_2}(V_2)$ is isomorphic to a subgroup of $\GL_2(p)$. Note that $S \nleq \C_{G_2}(V_2)$ (otherwise $V_2\leq Z$) and $E_2 \in \Syl_p(\C_{G_2}(V_2))$. So $\langle (S)^{G_2} \rangle/E_2$ acts non-trivially on $V_2$ and by Theorem \ref{class.3} we deduce that $O^{p'}(\Out_\F(E_2)) \cong \langle (S)^{G_2} \rangle/E_2 \cong \SL_2(p)$.

Suppose $|V_2|= p^3$.
\begin{itemize}
\item If $V_2 \nleq E_1$ then $S=E_1V_2$, since $[S \colon E_1] = p$ by Theorem \ref{class.3}. Also $[E_1,S] \nleq \Phi(E)$ by Lemma \ref{strict.frattini}, so $[E_1, V_2]\nleq \Phi(E_1)$. On the other hand, since $V_2$ is abelian and normal in $S$ we get $[E_1, V_2, V_2] = 1$. Thus $V_2\Phi(E_1)/\Phi(E_1)$ acts quadratically on $E_1/\Phi(E_1)$ and by Lemma   \ref{quad.action} applied with $T=\Phi(E_1)$ (that is $\F$-characteristic in $S$ because $E_1$ is $\F$-characteristic in $S$) we conclude that $O^{p'}(\Out_\F(E_1)) \cong \SL_2(p)$.

\item Assume $V_2 \leq E_1$. By the maximality of $T$ the group $V_2$ is not normalized by $G_1$. Hence there exists $g\in G_1 \backslash \N_{G_1}(S)$ such that $V_2 \neq V_2^g$. Note that $V_2^g \leq E_1$. If $[V_2^g , V_2]=1$ then $V_2V_2^g$ is an elementary abelian subgroup of $E_1$ and so $|V_2V_2^g|\leq p^3$ because $E_1$ has sectional rank $3$, contradicting the fact that $V_2 < V_2V_2^g$ and $|V_2| = p^3$.  Thus $[V_2^g, V_2] \neq 1$. In particular $V_2^g \nleq E_2$. On the other hand $[V_2,V_2^g]\leq V_2 \cap V_2^g$, so $[V_2, V_2^g, V_2^g] =1$. Hence $V_2^g$ acts quadratically on $V_2$ and by Lemma \ref{quad.action} we conclude that $O^{p'}(\Out_\F(E_2)) \cong \SL_2(p)$.

\end{itemize}
\end{proof}

\begin{theorem}\label{Op'.typeI}
Let $E_1 \leq S$ and $E_2 \leq S$ be distinct $\F$-essential subgroups of $S$. Suppose that $E_1$ and $E_2$ are $\F$-characteristic in $S$ and have sectional rank at most $3$.  Then, up to interchanging the definitions of $E_1$ and $E_2$, the following hold:

\begin{enumerate}
\item either $O^{p'}(\Out_\F(E_1)) \cong O^{p'}(\Out_\F(E_2)) \cong \SL_2(p)$;
\item or $O^{p'}(\Out_\F(E_1))\cong \PSL_2(p)$, $O^{p'}(\Out_\F(E_2))\cong \SL_2(p)$ and $S$ has rank $2$.
\end{enumerate}
\end{theorem}

\begin{proof}
By Theorem  \ref{SL2p}, we may assume that $O^{p'}(\Out_\F(E_2)) \cong \SL_2(p)$. 
Suppose the group $O^{p'}(\Out_\F(E_1))$ is not isomorphic to $\SL_2(p)$. Hence by Theorem \ref{class.3} the group $E_1$ has rank $3$ and either $O^{p'}(\Out_\F(E_1))\cong \PSL_2(p)$ or  $p=3$ and $O^{3'}(\Out_\F(E_1)) \cong 13 \colon 3$.
 By Lemma \ref{strict.frattini} we have $\Phi(E_1) < \Phi(S)$, so $[S \colon \Phi(S)]=p[E_1 \colon \Phi(S)] \leq p^3$ and $S$ has rank at most $3$. If $S$ has rank $3$, then $O^{p'}(\Out_\F(E_1)) \cong \SL_2(p)$ by Lemma \ref{norm.rank3.SL2p}, a contradiction. Therefore the group $S$ has rank $2$. 
It remains to show that $O^{p'}(\Out_\F(E_1))\cong \PSL_2(p)$.

Aiming for a contradiction, suppose that $O^{p'}(\Out_\F(E_1))$ is not isomorphic to $\PSL_2(p)$. Then $p=3$ and $O^{3'}(\Out_\F(E_1)) \cong 13 \colon 3$.
Let $\tau \in \Z(O^{p'}(\Out_\F(E_2)))$ be an involution and let $\ov{\tau} \in \Out_\F(S)$ be such that $\ov{\tau}|_{E_2}=\tau$ (that exists because $E_2$ is $\F$-essential and $S=\N_S(E_2)$).
Since $E_1$ is $\F$-characteristic in $S$, we deduce that $\ov{\tau}|_{E_1} \in \Out_\F(E_1)$.
So $\Out_\F(E_1) \cong (13 \colon 3) \times \C_{2}$ (as $13 \colon \C_6$ is not a subgroup of $\GL_3(3)$). In particular $\ov{\tau} \in \Z(\Out_\F(E_1))$ and so it acts trivially on the quotient $S/E_1$.  Note that $S=E_1E_2$ (because $[S \colon E_1] = [S \colon E_2]=p$) and by definition $\ov{\tau}$ acts trivially on $S/E_2 \cong E_1/(E_1\cap E_2)$. Consider the following series of $\F$-characteristic subgroups of $S$:
\[ \Phi(S) \leq E_1\cap E_2 \leq E_1 < S.\]
Since $S$ has rank $2$ we have $\Phi(S)=E_1 \cap E_2$ and by Lemma \ref{char.series} we deduce that $\ov{\tau}\in O_3(\Out_\F(S))$. However, $O_3(\Aut_\F(S)) = \Inn(S)$ since $S$ is fully normalized, and we get a contradiction.
\end{proof}

Note that the assumptions of Theorem \ref{Op'.typeI} are always satisfied when $S$ has sectional rank $3$ and there are two $\F$-essential subgroups that are $\F$-characteristic in $S$.

We conclude this section determining sufficient conditions for an $\F$-automorphism $\varphi$ of an $\F$-essential subgroup $E$ of $S$ of rank at most $3$ to be the restriction of an automorphism $\hat{\varphi}$ of a subgroup of $S$ properly containing $E$.

\begin{definition}\label{norm.tower} Let $E\leq S$ be an $\F$-essential subgroup of $S$. Set
\[ \N_S^0(E)=E \text{ and } \N_S^i(E)=\N_S(\N^{i-1}_S(E)) \text{ for every } i \geq 1.\]
We refer to the series
\[ E=\N^0_S(E) < \N^1_S(E) < \dots < \N^{m-1}_S(E) < \N^m_S(E)=S\]
as the normalizer tower of $E$ in $S$.
If $[\N^i_S(E) \colon \N^{i-1}_S(E)]=p$ for every $1\leq  i \leq m$ then we say that $E$ has maximal normalizer tower in $S$.

When it does not lead to confusion, we will write $\N^i$ in place of $\N^i_S(E)$.
\end{definition}

\begin{lemma}\label{rank.norm.tower}
Let $E\leq S$ be an $\F$-essential subgroup. If $E$ has maximal normalizer tower in $S$ and $[S \colon E] = p^m$, then for every $1\leq i \leq m$ we have
\[ \Phi(\N^{i-1})< \Phi(\N^i) \quad  \text{ and } \quad  \rm{rank}(\N^i) \leq \rm{rank}(\N^{i-1}).\]
\end{lemma}

\begin{proof}
Note that $E$ having maximal normalizer tower implies $\Phi(\N^i)\leq \N^{i-1}$ for every $i\geq 1$.
By Lemma \ref{strict.frattini} we have $\Phi(E) < \Phi(\N^1)$.
Suppose $2\leq i \leq m$.  If  $\Phi(\N^{i-1}) =\Phi(\N^i)$ then  $\Phi(\N^{i})\leq \N^{i-2}$ and so $\N^{i-2} \norm \N^i$. Thus by definition of the normalizer tower we get $\N^i=\N^{i-1}=S$, which is a contradiction.

Therefore for every $1 \leq i \leq m $ we have
$\Phi(\N^{i-1}) < \Phi(\N^i)$ and
\[ p^{\rm{rank}(\N^i)} =[\N^i \colon \Phi(\N^i)] = [\N^i \colon \N^{i-1}][\N^{i-1}\colon \Phi(\N^i)]\quad < \quad p[\N^{i-1} \colon \Phi(\N^{i-1})] = p^{\rm{rank}(\N^{i-1})+1}.\]
Hence $\rm{rank}(\N^i) \leq \rm{rank}(\N^{i-1})$.
\end{proof}

Let $E$ be an $\F$-essential subgroup of $S$ of rank at most $3$.
Since $E$ is fully normalized it is receptive and every $\F$-automorphism of $E$ that normalizes the group $\Aut_S(E)\cong \N_S(E)/\Z(E)$ is the restriction of an $\F$-automorphism of the group $\N_S(E)$. The following lemma gives sufficient conditions for a morphism $\varphi \in \N_{\Aut_\F(E)}(\Aut_S(E))$ to be the restriction of an $\F$-automorphism of $\N^j$, for some $j\geq 1$.

\begin{lemma}\label{lift.Ni}
Let $E$ be an $\F$-essential subgroup of $S$ of rank at most $3$.
Let $K\leq E$ be a subgroup of $E$ containing $[E,E]$ but not $[\N^1, \N^1]$.
Let $j\in \mathbb{N}$ be such that $\N^j \leq \N_S(K)$.
Then
\begin{enumerate}
\item $E$ has maximal normalizer tower in $\N^j$ and the members of this tower are the first $j$ members of the normalizer tower of $E$ in $S$;
\item if $P$ is a subgroup of $S$ containing $E$, then either $P=\N^i \leq \N^j$ for some $i\leq j$ or $\N^j < P$;
\item for every $1 \leq i\leq j-1$, if $K$ is $\F$-characteristic in  $\N^i$ then
 $\Aut_\F(\N^i) = \Aut_S(\N^i)\N_{\Aut_\F(\N^i)}(\N^{i-1})$, $\Aut_S(\N^i)\norm \Aut_\F(\N^i)$, $\N^i$ is not $\F$-essential and every morphism in $\Aut_\F(\N^i)$ is the restriction of a morphism in $\Aut_\F(\N^{i+1})$.
\end{enumerate}
In particular if  $K$ is $\F$-characteristic in  $\N^i$ for every $1\leq i\leq j-1$ then  every morphism in  $\N_{\Aut_\F(E)}(\Aut_S(E))$ is the restriction of an $\F$-automorphism of $\N^j$ that normalizes each member of the normalizer tower of $E$ in $\N^j$.
\end{lemma}

\begin{remark} Recall that $[\N^1 \colon E]=p$ by Theorem \ref{class.3}.
In particular $[\N^1,\N^1]\leq E$ and so $K < E$. The assumption $[E,E]\leq K$ implies that $\N^0 = E \leq \N_S(K)$. Finally note that $\N^j$ is a member of the normalizer tower of $E$, so $\N^i = \N_S^{i}(E) = \N_{\N^j}^i(E)$ for every $i\leq j-1$.
\end{remark}

The key idea for the proof of Lemma \ref{lift.Ni} is that the quotient group $E/K$ is a soft subgroup of $\N^j/K$, defined by H{\'e}thelyi in \cite{Het1} as an abelian self-centralizing subgroup having index $p$ in its normalizer.

\begin{proof}
Consider the group $\N^j/K$. Notice that the subgroup $E/K$ is abelian and for every $i\leq j$ we have  $\N^i(E/K) = \N^i/K$.
Since  $[\N^1, \N^1]\nleq K$ and $[\N^1 \colon E]=p$ by Theorem \ref{class.3}, we deduce that $E/K$ is self-centralizing in $\N^j/K$ and $[\N^1(E/K) \colon E/K] = p$.
Therefore $E/K$ is a soft subgroup of $\N^j/K$. In particular by \cite[Lemma 2]{Het1} the group $E$ has maximal normalizer tower in $\N^j$ and the members of such tower are the only subgroups of $\N^j$ containing $E$.

Let $P\leq S$ be a subgroup of $S$ containing $E$. If $P\leq \N^j$ then $P=\N_{\N^j}^i(E)=\N_S^i(E)=\N^i$ for some $i$. Suppose that $P \nleq \N^j$. We show that $\N^i\leq P$ for every $0\leq i \leq j$ by induction on $i$.
By assumption $\N^0=E\leq P$. Suppose $\N^i\leq P$ for some $0\leq i \leq j-1$. Note that $\N^i < \N_P(\N^i)\leq \N^{i+1}$ and since $[\N^{i+1} \colon \N^i]=p$ we deduce that $\N^{i+1}=\N_P(\N^i) \leq P$. Therefore $\N^i\leq P$ for every $0\leq i \leq j$ and so $\N^j\leq P$.

For every $1\leq i \leq j-1$,  let $H_i \in \N^{i-1}$ be such that $H_i/K =  \Z_i(\N^i/K)$ (the i-th center of $\N^i/K$). Also, let $H_j\leq H_{j-1}$ be such that $H_j/K = \Z(\N^1/K)[\N^j/K, \N^j/K]$. Then by \cite[Lemma 1 and Theorem 2]{Het2} we have that $\N^i/H_i \cong \C_p \times \C_p$ for every $1\leq i \leq j$ and $H_j/K$ is characteristic in $\N^j/K$. In particular $\Phi(\N^i) \leq H_i$ for every $1\leq i \leq j$.

Suppose that $K$ is $\F$-characteristic in $\N^i$ for some $1\leq i \leq j-1$. Then $H_i$ is $\F$-characteristic in $\N^i$ and the group $\Aut_\F(\N^i)$ acts on the quotient $\N^i/H_i\cong \C_p \times \C_p$. Since $\N^i < \N^j \leq S$, the group $\Aut_S(\N^i) \cong \N^{i+1}/\C_S(\N^i)$ acts non-trivially on the set of $\F$-conjugates of $\N^{i-1}$ contained in $\N^i$.  Note that $\N^i/H_i$ has $p+1$ maximal subgroups and at least $p$ of these are $\F$-conjugates of $\N^{i-1}$.
If $\alpha \in \Aut_\F(\N^i)$ then $E\alpha$ is an $\F$-essential subgroup of $S$ by Theorem \ref{class.3}, $\N^i=\N_S^i(E\alpha)$ and $K\alpha=K$. So by part (1) the group $E\alpha$ has maximal normalizer tower in $\N^{i+1}$. Thus $\N^i = \N_S(\N^{i-1}\alpha)$. Since $H_{i+1} \norm \N^{i+1}$, we deduce that $H_{i+1}$ is not of the form $\N^{i-1}\alpha$ for any $\alpha \in \Aut_\F(\N^i)$ and so $H_{i+1}$ is $\F$-characteristic in $\N^i$. Since $\N^i$ has rank at most $3$ by Lemma \ref{rank.norm.tower}, we deduce that  $[H_{i+1} \colon \Phi(\N^i)] \leq p$ and so  $\Aut_S(\N^i)$ stabilizes the series of subgroups  $\Phi(\N^i) \leq H_i < H_{i+1} < \N^i$. By Lemma \ref{char.series} we conclude that $\Aut_S(\N^i)\norm \Aut_\F(\N^i)$. In particular $O_p(\Out_\F(\N^i)) \neq 1$ and so $\N^i$ is not $\F$-essential. Also, the action of $\Aut_S(\N^i)$ on the conjugates of $\N^{i-1}$ contained in $\N^i$ is transitive and so by the Frattini Argument (\cite[3.1.4]{KS}) we have
\[\Aut_\F(\N^i) = \Aut_S(\N^i)\N_{\Aut_\F(\N^i)}(\N^{i-1}).\]
Note that the group $\N^i$ is $\F$-centric, because it contains $E$, and so it is fully centralized in $\F$. Since $\F$ is a saturated fusion system, we deduce that $\N^i$ is receptive (\cite[Theorem 5.2(2)]{RS}). Since $\Aut_S(\N^i)\norm \Aut_\F(\N^i)$ we conclude that  every morphism in $\Aut_\F(\N^i)$ is the restriction of a morphism in  $\Aut_\F(\N^{i+1})$.

The last statement follows from part (3) and the fact that $E$ is receptive and so every morphism in   $\N_{\Aut_\F(E)}(\Aut_S(E))$ is the restriction of a morphism in $\Aut_\F(\N^1)$.
\end{proof}

If $E$ is an abelian $\F$-essential subgroup of $S$ of rank at most $3$, then applying Lemma \ref{lift.Ni} with $K=1$ and $\N^j= S$ we get the following.

\begin{cor}\label{lift.Ni.ab}
Let $E$ be an abelian $\F$-essential subgroup of $S$ of rank at most $3$. Then $E$ has maximal normalizer tower in $S$ and it is not properly contained in any $\F$-essential subgroup of $S$. In particular every morphism in  $\N_{\Aut_\F(E)}(\Aut_S(E))$ is the restriction of an $\F$-automorphism of $S$ that normalizes each member of the normalizer tower of $E$ in $S$.
\end{cor}

\section{Properties of the $\F$-core and proof of  Theorem \ref{rank2pearl}}

\begin{definition}\label{core} Let $E_1 \leq S$ and $E_2 \leq S$ be $\F$-essential subgroups of $S$ such that $\N_S(E_1)=\N_S(E_2)$. We define the $\F$-core of $E_1$ and $E_2$, denoted $\coreF(E_1,E_2)$, as the largest subgroup $T$ of $E_1 \cap E_2$ that is normalized by $\Aut_\F(E_1)$, $\Aut_\F(E_2)$ and $\Aut_\F(\N_S(E_1))$. We set $\coreF(E_1)=\coreF(E_1,E_1)$ and we call it the $\F$-core of $E_1$.
\end{definition}

The structure of the $\F$-cores of $\F$-essential subgroups will play a crucial role in the proofs of most of the results of this paper (and we already used it in the proofs of Theorems \ref{frattini.is.char}, \ref{typeIIauto} and \ref{SL2p}).
In this section we describe the main properties of the $\F$-core.

\begin{remark}
If $E$ is an $\F$-essential subgroup of $S$, then $E=\coreF(E)$ if and only if $E$ is $\F$-characteristic in $S$. Indeed, if
$E=\coreF(E)$ then $E$ is $\F$-characteristic in $\N_S(E)$ and so $E$ is normal in $\N_S(\N_S(E))$, implying that $\N_S(\N_S(E)) = \N_S(E) = S$. Thus $E$ is $\F$-characteristic in $S$. On the other hand, if $E$ is $\F$-characteristic in $S$ then $S=\N_S(E)$ and so $E=\coreF(E)$.
\end{remark}

\begin{lemma}\label{same.T}
Let $E$ be an $\F$-essential subgroup of $S$ and set $T=\coreF(E)$. If $\alpha\in \Hom_\F(\N_S(E),S)$ then $T\alpha = \coreF(E\alpha)$.

In particular $\coreF(E)=\coreF(E,E\alpha)=\coreF(E\alpha)$ for every $\alpha \in \Aut_\F(\N_S(E))$.
\end{lemma}

\begin{proof}
If $E=E\alpha$ then $\alpha|_E \in \Aut_\F(E)$ and so $T\alpha=T=\coreF(E)$.

Suppose $E\neq E\alpha$.
Clearly $T\alpha $ is a subgroup of $E\alpha$. Note that $\Aut_\F(E\alpha) = \alpha^{-1} \Aut_\F(E) \alpha$. Since $E$ is an $\F$-essential subgroup of $S$, it is fully normalized in $\F$. Hence $[ \N_S(E)\alpha \colon E\alpha ]= [ \N_S(E) \colon E] \geq [\N_S(E\alpha) \colon E\alpha]$. Since $\N_S(E)\alpha\leq \N_S(E\alpha)$ we deduce that  $\N_S(E)\alpha = \N_S(E\alpha)$ and so $\Aut_\F(\N_S(E)\alpha) = \alpha^{-1} \Aut_\F(\N_S(E)) \alpha$.
It's now easy to see that $T\alpha=\coreF(E\alpha)$.

Assume $\alpha \in \Aut_\F(\N_S(E))$. Then $\N_S(E\alpha) = \N_S(E)$ and by maximality of $T$ we have $\coreF(E,E\alpha)\leq T$. On the other hand, $T=T\alpha = \coreF(E\alpha)$, so $T$ is contained in $E\cap E\alpha$ and is $\F$-characteristic in $E$, $E\alpha$ and $\N_S(E)$. Hence $T\leq \coreF(E,E\alpha)$, which implies $T=\coreF(E,E\alpha)$.
\end{proof}

\begin{remark} Lemma \ref{same.T} says in particular that if $E$ is an $\F$-essential subgroup of $S$ not $\F$-characteristic in $S$ then the $\F$-core of $E$ can always be described as the $\F$-core of two distinct $\F$-essential subgroups of $S$.
\end{remark}

\begin{lemma}\label{properties.T}
Let $E_1$ and $E_2$ be distinct $\F$-essential subgroups of $S$ such that $\N_S(E_1)=\N_S(E_2)$.
Set $\N=\N_S(E_1)=\N_S(E_2)$, $E_{12}=E_1 \cap E_2$ and $T=\coreF(E_1,E_2)$.
Suppose that $E_1$, $E_2$ and $T$ have rank at most $3$.
Then for every $1 \leq i \leq 2$ the following hold:
\begin{enumerate}
\item $\C_{E_i}(T) \nleq T$;
\item $\C_{\N}(T) \nleq E_{12}$;
\item if $\C_{\N}(T) \nleq E_i$ then $O^{p'}(\Out_\F(E_i))$ centralizes $T$.
\end{enumerate}
\end{lemma}

\begin{proof}
Let $G_i$ be a model for $\N_{\F}(E_i)$.
As an intermediate step we show that $\C_{G_i}(T)\nleq T$ for every $i\in \{1,2 \}$.
Suppose for a contradiction that $\C_{G_1}(T) \leq T$.
Then $\C_{\N}(T) \leq T$. In particular $\C_{\N}(T)= \C_{E_i}(T) \norm \C_{G_i}(T)$ for every $i\in \{1,2\}$.  Let $g\in \C_{G_2}(T)$. Note that $[E_2,g] \leq E_2 \cap \C_{G_2}(T) = \C_{E_2}(T) \leq T$. Thus $g$ stabilizes the sequence $1 < T < E_2$. Hence by Lemma \ref{char.series} and the fact that $E_2$ is $\F$-essential we deduce that $g \in E_2$, and so $g\in \C_{E_2}(T) \leq T$. Therefore $\C_{G_2}(T) \leq T$.
Hence we have $\C_{G_i}(T) \leq T$ for every $i\in \{1, 2\}$.
In particular by Lemma \ref{char.series} we deduce that for every $i\in \{1, 2\}$ the group $\C_{G_i}(T/\Phi(T))$ is a normal $p$-subgroup of $G_i$ and  so it is contained in  $E_i$.  Therefore $\C_{\N}(T/\Phi(T)) \leq E_1 \cap E_2$ and  $\C_{\N}(T/\Phi(T))=\C_{G_1}(T/\Phi(T))=\C_{G_2}(T/\Phi(T))$. By the maximality of $T$ and the fact that $T$ centralizes $T/\Phi(T)$ we conclude that
\[T = \C_{G_1}(T/\Phi(T)) = \C_{G_2}(T/\Phi(T)).\]

Thus the quotient $\N/(E_1 \cap E_2)$ acts non-trivially on $T/\Phi(T)$. By assumption the groups $E_1$ and $E_2$ have  rank at most $3$. Hence by Theorem \ref{class.3} we get $[\N\colon E_1] = [\N \colon E_2] = p$.
So $[\N \colon (E_1 \cap E_2)] =p^2$, that implies $[T \colon \Phi(T)] \geq p^3$. Note that $T$ is supposed to have rank at most $3$ and so $[T \colon \Phi(T)]=p^3$.
For every $i\in \{1,2\}$ the quotient $G_i/T$ is isomorphic to a subgroup of $\Aut(T/\Phi(T))\cong \GL_3(p)$, and so $\N/T \cong p^{1+2}_+$.  In particular $[ (E_1 \cap E_2) \colon T] = p$ and $E_i/T \cong \C_p \times \C_p$ for every $i\in \{1,2\}$.
Note that $[E_i, T] \Phi(T) \leq [\N,T]\Phi(T)$ for every $i\in \{1,2\}$ and $[\N,T] = [E_1,T][E_2,T]$.
If $[E_1,T]\Phi(T) = [E_2,T]\Phi(T)$ then  $[E_1,T]\Phi(T)= [\N,T]\Phi(T)$ and $\N/(E_1 \cap E_2)$ is isomorphic to a subgroup of $\Aut([\N,T]\Phi(T)/\Phi(T))$, a contradiction.
Thus $[E_1,T]\Phi(T) \neq [E_2,T]\Phi(T)$. In particular we get $[[\N,T]\Phi(T) \colon \Phi(T)] = p^2$.

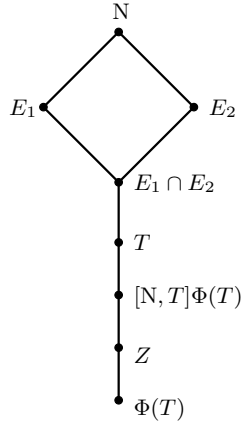
\begin{figure}[H]
\centering
\begin{tikzpicture}[x=1.00mm, y=1.00mm, inner xsep=0pt, inner ysep=0pt, outer xsep=0pt, outer ysep=0pt]
\path[line width=0mm] (-16.54,-43.38) rectangle +(33.71,59.52);
\definecolor{L}{rgb}{0,0,0}
\path[line width=0.30mm, draw=L] (0.,10) -- (-10,-0.0);
\path[line width=0.30mm, draw=L] (10,-0) -- (-0,-10);
\path[line width=0.30mm, draw=L] (10.0,-0.0) -- (-0.0,10.);
\path[line width=0.30mm, draw=L] (-0.0,-10.0) -- (-10.0,0);
\definecolor{F}{rgb}{0,0,0}
\path[line width=0.15mm, draw=L, fill=F] (-10,-0) circle (0.50mm);
\path[line width=0.15mm, draw=L, fill=F] (-0.0,10) circle (0.50mm);
\path[line width=0.15mm, draw=L, fill=F] (-0.0,-10.0) circle (0.50mm);
\path[line width=0.15mm, draw=L, fill=F] (10.0,-0.0) circle (0.50mm);
\path[line width=0.30mm, draw=L] (0.0,-10) -- (0.0,-39);
\draw(2,-19) node[anchor=base west]{\fontsize{8.54}{10.24}\selectfont $T$};
\draw(2,-11) node[anchor=base west]{\fontsize{8.54}{10.24}\selectfont $E_1\cap E_2$};
\draw(12,-1) node[anchor=base west]{\fontsize{8.54}{10.24}\selectfont $E_2$};
\draw(-14.54,-1) node[anchor=base west]{\fontsize{8.54}{10.24}\selectfont $E_1$};
\path[line width=0.15mm, draw=L, fill=F] (0,-18) circle (0.50mm);
\draw(-0.82,11.78) node[anchor=base west]{\fontsize{8.54}{10.24}\selectfont $\N$};
\path[line width=0.15mm, draw=L, fill=F] (0.,-25) circle (0.50mm);
\path[line width=0.15mm, draw=L, fill=F] (0.,-32) circle (0.50mm);
\path[line width=0.15mm, draw=L, fill=F] (0.00,-39) circle (0.50mm);
\draw(2,-41) node[anchor=base west]{\fontsize{8.54}{10.24}\selectfont $\Phi(T)$};
\draw(2,-26) node[anchor=base west]{\fontsize{8.54}{10.24}\selectfont $[\N,T]\Phi(T)$};
\draw(2,-34) node[anchor=base west]{\fontsize{8.54}{10.24}\selectfont $Z$};
\end{tikzpicture}%
\end{figure}

Let $Z$ be the preimage in $\N$ of $\Z(\N/\Phi(T))$. Then $Z \leq T$ and $[Z \colon \Phi(T)] = p$. Since $Z\leq [E_i,T]\Phi(T) \leq [\N,T]\Phi(T)$ for every $i$, we may assume  that
\[ [E_1,T]\Phi(T) = Z \text{ and } [E_2,T]\Phi(T) = [\N,T]\Phi(T).\]

Let $x \in (E_1 \cap E_2) \backslash T$ and let $t\in T$. Note that $[x,t] \in [E_1,T]\Phi(T) =Z$, so $[x,t]$ commutes with $t$ and $x$ modulo $\Phi(T)$. Hence by  properties of commutators (\cite[Lemma 2.2.2]{GOR}) we have
\[ (xt)^p = t^px^p[x,t]^\frac{p(p-1)}{2} = x^p \mod \Phi(T).\]
Since $E_1\cap E_2=  \langle x \rangle T$ we deduce that $(E_1 \cap E_2)^p \Phi(T) = \langle x^p \rangle \Phi(T)$. Thus $ (E_1\cap E_2)^p\Phi(T)=Z$ and the quotient $(E_1 \cap E_2)/ Z$ is elementary abelian of order $p^3$.

Note that $\Phi(E_1) \leq T$ and so either $E_1$ has rank $2$ (and $T=\Phi(E_1)$) or $[T \colon \Phi(E_1)] = p$. In particular by Theorem \ref{class.3} we have $\langle (\N)^{G_1} \rangle/E_1 \cong \SL_2(p)$.
Also, $G_1$ acts transitively on the maximal subgroups of $E_1$ containing $T$ and normalizes $[T,E_1]\Phi(T)$. Hence we conclude that $E_1/[T,E_1]\Phi(T) = E_1/Z$ has exponent $p$.

Let $\tau \in \langle (\N)^{G_1} \rangle$ be an involution that inverts $E_1/T$. Note that $T/Z$ is a natural $\SL_2(p)$-module for  $\langle (\N)^{G_1} \rangle/E_1$ (otherwise $\langle (\N)^{G_1} \rangle$ would centralize every quotient of two consecutive subgroups in the series $\Phi(T) < Z < T$ and so $\langle (\N)^{G_1} \rangle$ would be a $p$-group, a contradiction).
Hence $\tau$ inverts the quotient $T/Z$. Thus $\tau$ inverts every quotient of two consecutive subgroups in the series
\[ \Z < [\N,T]\Phi(T) < T < E_1 \cap E_2 < E_1.\]
Therefore the group $E_1/Z$ is abelian and so elementary abelian of order $p^4$. Thus $\Phi(E_1)\leq Z$ and $E$ has rank at least $4$, a contradiction.

We proved that $\C_{G_i}(T)\nleq T$ for every $i$.
Now suppose for a contradiction that $\C_{E_i}(T) \leq T$ for some $i$.
Then $\C_{G_i}(T)$ is a normal subgroup of $G_i$ not contained in $E_i=O_p(G_i)$. Hence $\C_{G_i}(T)$ is not a $p$-group and there exists a non trivial element $g\in \C_{G_i}(T)$ of order prime to $p$. Note that the direct product $\langle g \rangle  \times T$ acts by conjugation on $E_i$. Then by \cite[Theorem 5.3.4]{GOR} we get $[g, \C_{E_i}(T)] \neq 1$, contradicting the fact that $\C_{E_i}(T)\leq T \leq \C_{G_i}(g)$.
Thus $\C_{E_i}(T) \nleq T$ for every $i$.

Suppose for a contradiction that $\C_{\N}(T) \leq E_1 \cap E_2$. Then $\C_{\N}(T)=\C_{E_1}(T)=\C_{E_2}(T)$ is $\F$-characteristic in $E_1$, $E_2$ and $\N$ and by maximality of $T$ we conclude $\C_{E_i}(T)\leq T$, contradicting  what we proved above.

Finally, assume that $\C_{\N}(T) \nleq E_i$ for some $i$. Then $\N=E_i\C_{\N}(T)$, since $[\N \colon E_i] = p$ by Theorem \ref{class.3}. In particular
$\Out_S(E_i) \cong \N/E_i \cong \C_{\N}(T)/\C_{E_i}(T)$ centralizes $T$. Hence $O^{p'}(\Out_\F(E_i)) = \langle \Out_S(E_i)^{\Out_\F(E_i)}\rangle$ centralizes $T$.
\end{proof}

\begin{theorem}\label{prop.rank.2}
Suppose $p$ is an odd prime,  $S$ is a $p$-group and $\F$ is a saturated fusion system on $S$. Let $E$ be an $\F$-essential subgroup of $S$ such that
\begin{itemize}
\item $E$ has rank $2$;
\item $\Phi(E)$ has rank at most $3$; and
\item $E$ is not $\F$-characteristic in $S$.
\end{itemize}
Then $E$ is an $\F$-pearl.
\end{theorem}

\begin{proof}
By Lemma \ref{frattini.is.char} we get $\coreF(E)=\Phi(E)$. Let $\alpha\in \Aut_\F(\N_S(E))$ be a morphism that does not normalize $E$. Then by Lemma \ref{same.T} we have $\Phi(E)=\coreF(E,E\alpha)$ and by Lemma \ref{properties.T} applied with $E_1 = E$ and $E_2=E\alpha$ we conclude that $\C_E(\Phi(E)) \nleq \Phi(E)$. Since $[E \colon \Phi(E)]=p^2$ and $E$ is $\F$-essential, by Lemma \ref{char.series} we get $E=\Phi(E)\C_E(\Phi(E))$ and so $E=\C_E(\Phi(E))$. Thus $\Phi(E)\leq \Z(E)$ and by Theorem \ref{incenterfrat.is.pearl} we deduce that $E$ is an $\F$-pearl.
\end{proof}

\begin{proof}[\textbf{Proof of Theorem \ref{rank2pearl}}]
If $S$ has sectional rank $3$ then every subgroup of $S$ has rank at most $3$. Hence Theorem \ref{rank2pearl} is a direct consequence of Theorem \ref{prop.rank.2}.
\end{proof}

\begin{lemma}\label{T.strict.intersection}
Let $E_1$ and $E_2$ be distinct $\F$-essential subgroups of $S$ such that $\N_S(E_1)=\N_S(E_2)$.
Set $T=\coreF(E_1,E_2)$.
Suppose that $E_1$, $E_2$ and $T$ have rank at most $3$.
Then for every $1\leq i\leq 2$ either $T\leq \Phi(E_i)$ or $O^{p'}(\Out_\F(E_i))\cong \SL_2(p)$, $[T\Phi(E_i) \colon \Phi(E_i)]=p$ and \[T\Phi(E_i)/\Phi(E_i) = \C_{E_i/\Phi(E_i)}(O^{p'}(\Out_\F(E_i))).\]

\end{lemma}

\begin{proof}
Fix $1 \leq i \leq 2$ and set $E=E_i$ and $\N= \N_S(E)$.
Note that $\Phi(E)T$ is a proper $\F$-characteristic subgroup of $E$.
If the action of $\Out_\F(E)$ on $E/\Phi(E)$ is irreducible, then we have $\Phi(E)T=\Phi(E)$, and so $T\leq \Phi(E)$.
Suppose the action is reducible. Then $[E \colon \Phi(E)] = p^3$ and by Theorem \ref{class.3} we get that $\Out_\F(E)$ is isomorphic to  a subgroup of $\GL_2(p) \times \GL_1(p)$ and $O^{p'}(\Out_\F(E))\cong \SL_2(p)$. Let $\tau \in O^{p'}(\Out_\F(E))$ be an involution. Then by coprime action we have
\[ E/\Phi(E) \cong \C_{E/\Phi(E)}(\tau) \times [E/\Phi(E), \tau].\]
Note that the groups  $\C_{E/\Phi(E)}(\tau)$ and $[E/\Phi(E), \tau]$ are the only subgroups of $E/\Phi(E)$ that are normalized by $O^{p'}(\Out_\F(E))$. Thus
$\C_{E/\Phi(E)}(\tau) = \C_{E/\Phi(E)}(O^{p'}(\Out_\F(E)))$ and
 $[E/\Phi(E), \tau]=[E/\Phi(E), O^{p'}(\Out_\F(E))]$.
Also, either $T\leq \Phi(E)$ or $T\Phi(E)$ is the preimage in $E$ of one of these two subgroups of $E/\Phi(E)$.

It remains to prove that $T\Phi(E)$ cannot be the preimage in $E$ of the commutator group $[E/\Phi(E), O^{p'}(\Out_\F(E))]$. Suppose for a contradiction that it is.
Then $T/(T\cap \Phi(E)) \cong T\Phi(E)/\Phi(E)$ is a natural $\SL_2(p)$-module for  $O^{p'}(\Out_\F(E))$.  So $O^{p'}(\Out_\F(E))$ does not centralize $T$ and, by Lemma \ref{properties.T}, we have $\C_{\N}(T) \leq E$.
Since $T\Phi(E) \leq E_1 \cap E_2$ and $[E \colon \Phi(E)] = p^3$, we deduce that $T\Phi(E)=E_1 \cap E_2$.
Let $j\neq i$, $1\leq j\leq 2$. Then $\C_{E_j}(T) \leq \C_{\N}(T) \leq E$ and
\[ \C_{E_j}(T) = \C_{\N}(T) \cap E_j = \C_{E}(T) \cap T\Phi(E).\]
Thus $\C_{E_j}(T)$ is $\F$-characteristic in $E$.
Moreover $\Phi(E)T=\Phi(\N)T$, so $\C_{E_j}(T)=\C_{\N}(T) \cap \Phi(\N)T$ is $\F$-characteristic in $\N$.
Clearly $\C_{E_j}(T)$ is $\F$-characteristic in $E_j$ and we get  $\C_{E_j}(T)\leq T$ by the maximality of $T$, contradicting Lemma \ref{properties.T}.
Thus either $T\leq \Phi(E)$ or $T\Phi(E)$ is the preimage in $E$ of $\C_{E/\Phi(E)}(O^{p'}(\Out_\F(E)))$.
\end{proof}

\begin{lemma}\label{prop.T.II}
Let $E$ be an $\F$-essential subgroup of $S$ not $\F$-characteristic in $S$ and set $\N=\N_S(E)$ and $T=\coreF(E)$. Suppose that $E$ and $T$ have rank at most $3$. Then
\begin{enumerate}
\item $\C_E(T) \nleq T$;
\item $\C_{\N}(T) \nleq E$ and $\N=E\C_{\N}(T)$;
\item $O^{p'}(\Out_\F(E))$ centralizes $T$;
\item either $T\leq \Phi(E)$ or $T/\Phi(E) = \C_{E/\Phi(E)}(O^{p'}(\Out_\F(E)))$ and $[T\Phi(E) \colon \Phi(E)]=p$.
\end{enumerate}
 \end{lemma}

\begin{proof}
By Lemma \ref{same.T} we have $T=\coreF(E,E\alpha)$, for some $\alpha\in \Aut_\F(\N)$ such that $E\neq E\alpha$. Therefore we can apply Lemmas \ref{properties.T} and \ref{T.strict.intersection} with $E_1=E$ and $E_2=E\alpha$ and so statements $(1)$, $(3)$ and $(4)$ hold.
For part $(2)$, since $\C_{\N}(T)\nleq E \cap E\alpha$ and $\C_{\N}(T)$ is $\F$-characteristic in $\N$, we get that $\C_{\N}(T)\nleq E$. Finally since $E$ has rank at most $3$ by Theorem \ref{class.3} we have $[\N \colon E] = p$ and so $\N=E\C_{\N}(T)$.
\end{proof}

We end this section proving that under certain conditions the $\F$-core $T$ of $E_1$ and $E_2$ is either cyclic or isomorphic to the group $\C_{p^a} \times \C_p$, for some $a\geq 1$. We will see in Section 3 that these conditions are always satisfied when $E_1$ is an $\F$-essential subgroup of $S$ not $\F$-characteristic in $S$ such that $\N_S(E_1)$ has sectional rank at most $3$ and $T=\coreF(E_1)$.

\begin{theorem}\label{char.T}
Let $E_1$ and $E_2$ be distinct $\F$-essential subgroups of $S$ such that $\N_S(E_1)=\N_S(E_2)=\N$.
Set $E_{12}=E_1 \cap E_2$ and $T=\coreF(E_1,E_2)$.
Suppose that the following hold:
\begin{enumerate}
\item $E_1$, $E_2$ and $T$ have rank at most $3$;
\item $O^{p'}(\Out_\F(E_1)) \cong \SL_2(p)$ centralizes $T$;
\item there exists a subgroup $V\leq E_1$ that is $\F$-characteristic in $E_1$, has sectional rank at most $3$, is contained in  $\C_{E_1}(T)T$ and is such that  $V/T$ is a natural $\SL_2(p)$-module for $O^{p'}(\Out_\F(E_1))$.
\end{enumerate}
Then
$T$ is abelian, $T\leq \Z(V), |[V,V]|\leq p$ and
the group $T/[V,V]$  is cyclic.
\end{theorem}

\begin{proof}
Set $E=E_1$.
Since $V\leq \C_E(T)T$, we get \[\C_V(T)T= (V \cap \C_E(T))T = V \cap \C_E(T)T = V.\]
Note that $\C_V(T) \cap T = \Z(T)$ and so
\[ V/\Z(T) \cong T/\Z(T) \times \C_V(T)/\Z(T).\]

Since $\C_V(T)/\Z(T) \cong V/T \cong \C_p \times \C_p$ and $V$ has sectional rank $3$, we deduce that $T/\Z(T)$ has to be cyclic and so the group $T$ is abelian.
In particular $V=\C_V(T)T=\C_V(T)$.
Hence $T\leq \Z(V)$ and since $[V \colon T] = p^2$ we conclude that $|[V,V]|\leq p$.

Let $\tau \in O^{p'}(\Out_\F(E))$ be an involution. Then by assumption $\tau$ acts on $V$ and $T$ is the centralizer in $V$ of $\tau$. Thus by coprime action we get

\[ V/[V,V] \cong T/[V,V] \times [V/[V,V] , \tau].\]

Since $[V/[V,V], \tau] \cong \C_p \times \C_p$ and $V$ has sectional rank $3$, we deduce that the group $T/[V,V]$ is cyclic.
\end{proof}

\section{Structure of $\F$-essential subgroups that are not $\F$-characteristic in $S$}

Throughout this section, we assume the following hypothesis.

\begin{hp} Suppose that $p$ is an odd prime, $S$ is a $p$-group, $\F$ is a saturated fusion system on $S$ and $E\leq S$ is an $\F$-essential subgroup of $S$ not $\F$-characteristic in $S$ such that the group $\N_S(E)$ has sectional rank $3$. Set $T=\coreF(E)< E$.
\end{hp}

By assumption every subgroup of $\N_S(E)$ has rank at most $3$. So in particular $E$ has rank at most $3$ and by Theorem \ref{typeIIauto} we know that $O^{p'}(\Out_\F(E))\cong \SL_2(p)$. In this section we describe the structure of $E$.
We intend to apply Stellmacher's Pushing Up Theorem  (\cite[Theorem $1$]{stell}), stated in Theorem \ref{stell.here} of this paper. We first show that  the quotient group $\N_S(E)/T$  is non-abelian.

\begin{lemma}\label{non.abelian.quotient}
The quotient group $\N_S(E)/T\Phi(E)$ is non-abelian.
\end{lemma}

\begin{proof} Consider
the following series of $\F$-characteristic subgroups of $E$:
\[ \Phi(E) \leq T\Phi(E) < E.\]
By Lemma \ref{prop.T.II} we have  $[T\Phi(E) \colon \Phi(E)] \leq p$.
So $\N_S(E)$ centralizes the quotient $T\Phi(E)/\Phi(E)$. Since $O_p(\Aut_\F(E))=\Inn(E) \neq \Aut_S(E)$, by Lemma \ref{char.series} the group $\N_S(E)$ cannot centralize the quotient $E/T\Phi(E)$. Thus the quotient group $\N_S(E)/T\Phi(E)$ is not abelian.
\end{proof}

\begin{theorem}\label{characterization.E}
Set $V=[E,O^{p'}(\Aut_\F(E))]T$.
Then

\begin{minipage}{0.7\textwidth}
\begin{enumerate}\setlength\itemsep{1em}
\item $V/T$ is a natural $\SL_2(p)$-module for the group $O^{p'}(\Out_\F(E))\cong \SL_2(p)$;
\item $\N_S(E)/T$ has exponent $p$;
\item $E/T$ is elementary abelian and $p^2 \leq [E \colon T] \leq p^3$;
\item $[E/T \colon \Z(\N_S(E)/T)] =p$;
\item $T$ is abelian, $T\leq \Z(V)$, $|[V,V]|\leq p$ and  $T/[V,V]$ is a cyclic group.
\end{enumerate}

Moreover, if $[E \colon T] = p^2$, then $T \leq \Z(\N_S(E))$.
\end{minipage}
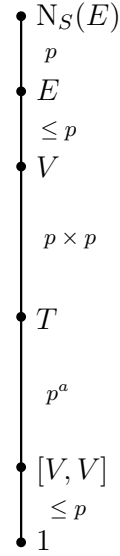
\begin{minipage}{0.35\textwidth}
\begin{figure}[H]
\centering
\begin{tikzpicture}[x=1.00mm, y=1.00mm, inner xsep=0pt, inner ysep=0pt, outer xsep=0pt, outer ysep=0pt]
\definecolor{L}{rgb}{0,0,0}
\path[line width=0.30mm, draw=L] (-0,30) -- (-0,-40);
\definecolor{F}{rgb}{0,0,0}
\path[line width=0.30mm, draw=L, fill=F] (-0,-10) circle (0.50mm);
\path[line width=0.30mm, draw=L, fill=F] (-0,-30) circle (0.50mm);
\path[line width=0.30mm, draw=L, fill=F] (-0,20) circle (0.50mm);
\path[line width=0.30mm, draw=L, fill=F] (-0,30) circle (0.50mm);
\draw(2.00,-11.65) node[anchor=base west]{\fontsize{11.38}{13.66}\selectfont $T$};
\draw(2.00,-31.63) node[anchor=base west]{\fontsize{11.38}{13.66}\selectfont $[V,V]$};
\draw(2.00,-41.07) node[anchor=base west]{\fontsize{11.38}{13.66}\selectfont $1$};
\draw(2.00,18.89) node[anchor=base west]{\fontsize{11.38}{13.66}\selectfont $E$};
\draw(2.00,28.88) node[anchor=base west]{\fontsize{11.38}{13.66}\selectfont $\N_S(E)$};
\draw[F] (3.35,-21.10) node[anchor=base west]{\fontsize{8.54}{10.24}\selectfont \textcolor[rgb]{0, 0, 0}{$p^a$}};
\draw[F] (3.82,-36.41) node[anchor=base west]{\fontsize{8.54}{10.24}\selectfont \textcolor[rgb]{0, 0, 0}{$\leq p$}};
\draw[F] (2.56,14.10) node[anchor=base west]{\fontsize{8.54}{10.24}\selectfont \textcolor[rgb]{0, 0, 0}{$\leq p$}};
\draw[F] (3.19,24.38) node[anchor=base west]{\fontsize{8.54}{10.24}\selectfont \textcolor[rgb]{0, 0, 0}{$p$}};
\path[line width=0.30mm, draw=L, fill=F] (-0,10) circle (0.50mm);
\path[line width=0.30mm, draw=L, fill=F] (-0,-40.00) circle (0.50mm);
\draw(2.00,8.63) node[anchor=base west]{\fontsize{11.38}{13.66}\selectfont $V$};
\draw[F] (3.07,-0.34) node[anchor=base west]{\fontsize{8.54}{10.24}\selectfont \textcolor[rgb]{0, 0, 0}{$p\times p$}};
\end{tikzpicture}%
\end{figure}
\end{minipage}
\end{theorem}

\begin{proof} Set $\N=\N_S(E)$,
let $G$ be a model for $\N_\F(E)$ and let $A=\langle {\N}^G \rangle \leq G$. Then $A/E \cong O^{p'}(\Out_\F(E))\cong \SL_2(p)$ by Theorem \ref{typeIIauto}.
We want to apply Stellmacher's Pushing Up Theorem (Theorem \ref{stell.here}) to the group $A/T$ and to its Sylow $p$-subgroup $\N/T$. Note that the quotient $\N/T$ is non-abelian by Lemma \ref{non.abelian.quotient}.

Let $T \leq W \leq \N$ be such that $W/T$ is characteristic in  $\N/T$ and $W/T \norm A/T$. Then $W \leq E=O_p(A)$, $W$ is $\F$-characteristic in $\N$ and $W \norm \N_G(\N)A=G$, that implies $W$ $\F$-characteristic in $E$. By the definition of $\F$-core, the group $T$ is the largest subgroup of $E$ that is $\F$-characteristic in $E$ and $\N$. So $W = T$ and $W/T=1$. 
Thus by Stellmacher's Pushing Up Theorem (Theorem \ref{stell.here}) and the fact that $\N$ has sectional rank $3$, we get that $V/T \leq \Z(E/T)$ and $V/T$ is a natural $\SL_2(p)$-module for $A/E$.
In particular $[E \colon T] \geq p^2$.

Let $\Omega_N\leq \N$ be the preimage in $\N$ of $\Omega_1(\Z(\N/T))$ and let $\Omega_E$ be the preimage in $E$ of $\Omega_1(\Z(E/T))$. Then Stellmacher's Pushing Up Theorem (Theorem \ref{stell.here}) tells us that $\N/\Omega_N$ is elementary abelian. Since $\N$ has sectional rank $3$ we deduce that $[\N \colon \Omega_N]\leq p^3$.

If $\Omega_\N\nleq E$ then $\N=E\Omega_N$ and
\[ [\N,\N]=[E,\N]=[E,E][E,\Omega_\N] \leq \Phi(E)T,\]
contradicting the fact that $N/T\Phi(E)$  is non-abelian by Lemma \ref{non.abelian.quotient}. Therefore $\Omega_\N\leq E$ and so $\Omega_N \leq \Omega_E$. By maximality of $T$, we also have $\Omega_\N \neq \Omega_E$. In particular
\[ [E \colon \Omega_E] < [E \colon \Omega_\N] = p^{-1}[\N \colon \Omega_\N] \leq p^2.\]
Therefore $[E \colon \Omega_E] \leq p$, which implies that $E/T$ is abelian.

Let $\tau\in O^{p'}(\Out_\F(E))$ be an involution and let $C\leq E$ be the preimage in $E$ of $\C_{E/T}(\tau) $. Then by coprime action we get
\[E/T\cong C/T \times [E/T, \tau].\]
Note that $[E/T, \tau] \leq V/T$ and since $V/T$ is a natural $\SL_2(p)$-module for $O^{p'}(\Out_\F(E))$, we deduce that $E/T \cong C/T \times V/T$. Thus $\N/C $ is isomorphic to a Sylow $p$-subgroup of  the group $ (\C_p \times \C_p) \colon \SL_2(p)$. Hence $\N/C \cong p^{1+2}_+$ and so $\N^p \leq C$. Therefore $\N^pT$ is a subgroup of $E$ that is $\F$-characteristic in $\N$ and normalized by $G=A\N_G(\N)$.
By maximality of $T$ we get $\N^p \leq T$. Hence $\N/T$ has exponent $p$ and $E/T$ is elementary abelian. In particular $[E \colon T] \leq p^3$.

Since $E/T$ is elementary abelian we have $\Omega_N/T =\Omega_1(\Z(\N/T))=\Z(\N/T)$.
Let $\alpha$ be an $\F$-automorphism of $\N$ such that $E\neq E\alpha$. Then $\N=EE\alpha$ and $E\alpha/T\cong E/T$ is abelian. Hence $\Omega_N= E \cap E\alpha$ and $[E \colon \Omega_N] =p$.

Part $(5)$  is a consequence of Theorem \ref{char.T}, once we have shown that  $V \leq \C_E(T)T$.
Note that the group $\C_E(T)T$ is an $\F$-characteristic subgroup of $E$ not contained in $T$ (by Lemma \ref{prop.T.II}).
Since $\Phi(E) \leq T$ and $E$ has rank at most $3$, by Lemma \ref{char.series} either $V\leq \C_E(T)T$ or $T=\Phi(E)$ and $[\C_E(T)T \colon T]=p$. Suppose for a contradiction that the latter holds.
Since $\C_E(T)T\norm \N$ we deduce $\C_E(T)T\leq \Omega_N$. Also $\C_E(T)T=(\C_{\N}(T) \cap E)T = \C_{\N}(T)T \cap E$. Therefore $\C_E(T)T =\C_{\N}(T)T \cap \Omega_N$. In particular $\C_E(T)T$ is normalized by $\Aut_\F(E)$ and $\Aut_\F(\N)$, contradicting the maximality of $T$.
Therefore $V\leq \C_E(T)T$ and we conclude by Theorem \ref{char.T}.

Finally, if $[E\colon T] = p^2$ then $V=E$ so $E=\C_E(T)$.
Since $\N = E\C_{\N}(T)$ by Lemma \ref{prop.T.II}, we deduce that $\N=\C_{\N}(T)$ and  so $T\leq \Z(\N)$.
\end{proof}

\begin{lemma}\label{lift.to.N2}
Suppose $\N_S(E)<S$ and set $\N^1 = \N_S(E)$ and $\N^2 = \N_S(\N^1)$.
 Then $[\N^2 \colon \N^1] = p$, $\Aut_\F(\N^1) = \Aut_S(\N^1)\N_{\Aut_\F(\N^1)}(E)$, $\N^1$ is not $\F$-essential and every automorphism of $E$ contained in  $\N_{\Aut_\F(E)}(\Aut_S(E))$ is the restriction to $E$ of an $\F$-automorphism of $\N^2$.
\end{lemma}

\begin{proof}
By Theorem \ref{characterization.E}(3) we have $[E,E]\leq T$ and by Lemma \ref{non.abelian.quotient}  we get $[\N^1,\N^1]\nleq T$.
Also note that $T\norm \N^2$ because $T$ is $\F$-characteristic in $\N^1$. Hence the statement follows from Lemma \ref{lift.Ni} applied with $K=T$ and $\N^j=\N^2$.
\end{proof}

\begin{lemma}\label{centerN2}
Suppose that $\N_S(E)<S$ and set $\N^1 = \N_S(E)$ and $\N^2 = \N_S(\N^1)$. Then $|\Z(\N^2/T)|=p$.
\end{lemma}

\begin{proof}
First notice that $T\norm \N^2$ so we can consider the group $\N^2/T$. If $[E \colon T] =p^2$ then the fact that $E$ is not normal in $\N^2$ implies that $\Z(\N^2/T) < E/T$ and so $|\Z(\N^2/T)|=p$. Hence by Theorem \ref{characterization.E} we can assume that $[E \colon T]=p^3$.
Let $C$ be the preimage in $E$ of the group $\C_{E/T}(O^{p'}(\Aut_\F(E)))$. Then $|C/T|=p$. Recall that $T$ is the largest subgroup of $E$ that is $\F$-characteristic in $E$ and $\N^1$. Hence $C$ is not $\F$-characteristic in $\N^1$.
By Lemma \ref{lift.to.N2} we have $\Aut_\F(\N^1)=\Aut_S(\N^1)\N_{\Aut_\F(\N^1)}(E)$. Since $\Aut_S(\N^1) \cong \N^2/\Z(\N^1)$ and $C$ is $\F$-characteristic in $E$, we deduce that $C$ is not normal in $\N^2$. In particular $C/T\nleq \Z(\N^2/T)$.
Since $C/T \leq \Z(\N^1/T)$ we get $\Z(\N^2/T) < \Z(\N^1/T)$. By Theorem \ref{characterization.E}(4) we have $|\Z(\N^1/T)|=p^2$ and so $|\Z(\N^2/T)|=p$.\end{proof}

We conclude this section with further properties of the quotient group $\N_S(E)/\Phi(E)$.

\begin{lemma}\label{Zindexp}
Let $Z$ be the preimage in $\N_S(E)$ of $\Z(\N_S(E)/\Phi(E))$.
Then  $Z$ is the preimage in $E$ of  $\Z(\N_S(E)/T)$. In particular $[E\colon Z] = p$, $Z$ is $\F$-characteristic in $\N_S(E)$ and $\Phi(\N_S(E)) \leq Z$.
\end{lemma}

\begin{proof}
Suppose that $\Phi(E) \neq T$. Then by Theorem  \ref{characterization.E}(3) we deduce that $\Phi(E) < T$, $E$ has rank $3$ and $[E \colon T] =p^2$. Also, by Lemma \ref{prop.T.II}(4) we have $T/\Phi(E)=\C_{E/\Phi(E)}(O^{p'}(\Aut_\F(E)))$.
In particular by coprime action we have $E/\Phi(E) \cong T/\Phi(E) \times [E/\Phi(E), O^{p'}(\Aut_\F(E))]$ and so
\[T/\Phi(E) \cap [E/\Phi(E), O^{p'}(\Aut_\F(E))] = 1.\]
 Since both $T/\Phi(E)$ and $ [E/\Phi(E), O^{p'}(\Aut_\F(E))]$ are normal subgroups of  $\N_S(E)/\Phi(E)$, their intersection with the center $Z/\Phi(E)$ is non-trivial. Hence \[|Z/\Phi(E)|\geq p^2.\] Since $Z \leq E$ and $[\N_S(E), E] \nleq \Phi(E)$ by Lemma \ref{strict.frattini}, we conclude that  $Z < E$ and so $|Z/\Phi(E)|= p^2$. In other words $[E \colon Z] =p$ and since $Z/T \leq \Z(\N_S(E)/T) < E/T$ we conclude that $Z/T = \Z(\N_S(E)/T)$.

Therefore in any case we have that $Z$ is the preimage in $E$ of  $\Z(\N_S(E)/T)$.
Note that $Z$ is $\F$-characteristic in $\N_S(E)$ because $T$ is $\F$-characteristic in $\N_S(E)$ and $[E \colon Z] =p$ by Theorem \ref{characterization.E}(4). So $[\N_S(E) \colon Z] =p^2$ and since $E$ is not $\F$-characteristic in $\N_S(E)$ we deduce that the quotient $\N_S(E)/Z$ is not cyclic. Hence $\N_S(E)/Z$ is elementary abelian and $\Phi(\N_S(E))\leq Z$.
\end{proof}

\begin{lemma}\label{same.rank}
The group $\N_S(E)/\Phi(E)$ has exponent $p$. In particular $\Phi(\N_S(E))=[\N_S(E),\N_S(E)]\Phi(E)$ and the groups $E$ and $\N_S(E)$ have the same rank.
\end{lemma}

\begin{proof}
If $\Phi(E) = T$ then $\N_S(E)/\Phi(E)$ has exponent $p$ by Theorem \ref{characterization.E}(2). Suppose $\Phi(E)\neq T$.
Then by Theorem  \ref{characterization.E}(3) we deduce that $\Phi(E) < T$, $E$ has rank $3$, $[E \colon T] =p^2$ and $T\leq \Z(\N_S(E))$.
Let $Z$ be the preimage in $E$ of $\Z(\N_S(E)/T)$. Then by Theorem \ref{characterization.E}(4) we have $[E \colon Z]=p$. Thus $[Z \colon T]=p$ and $Z$ is abelian.

By Theorem \ref{typeIIauto} we have  $O^{p'}(\Out_\F(E))\cong \SL_2(p)$ and by Theorem \ref{characterization.E} the quotient $E/T$ is a natural $\SL_2(p)$-module for $O^{p'}(\Out_\F(E))$. Thus there exists a morphism $\tau\in O^{p'}(\Aut_\F(E))$ that acts on $E/T$ as $\begin{pmatrix} -1 & 0 \\ 0 & -1 \end{pmatrix}$ with respect to the basis $\{ xT, zT\}$, for some  $x\in E \backslash Z$ and $z\in Z \backslash T$.
Then  $z\tau = z^{-1}t$ for some $t\in T$ and $\tau$ centralizes $T$  by Lemma \ref{prop.T.II}(4). Since $z^p \in T$ we get
\[ z^p = (z^p)\tau = (z\tau)^p =(z^{-1}t)^p = (z^p)^{-1}t^p.\]
Hence $z^p \in T^p$ and we conclude that $Z^p = T^p$. Since $Z$ and $T$ are abelian we also have $|Z|=|\Omega_1(Z)||Z^p|$ and $|T|=|\Omega_1(T)||T^p|$. Therefore 
\begin{equation}\tag{$\star$}\label{star} \Omega_1(T) < \Omega_1(Z) \leq \Omega_1(E).\end{equation}

Suppose for a contradiction that $\Omega_1(\N_S(E)) \leq E$. Then $\Omega_1(\N_S(E))= \Omega_1(E)$ is $\F$-characteristic in both $E$ and $\N_S(E)$ and so $E \leq T$ by the definition of $\F$-core. So $\Omega_1(E)\leq \Omega_1(T)$, contradicting (\ref{star}). So $\Omega_1(\N_S(E)) \nleq E$. Since $[\N_S(E) \colon E]=p$ by Theorem \ref{class.3}, we deduce that there exists an element $h\in \N_S(E)$ of order $p$ such that every element $g$ of $\N_S(E)$ can be written as a product $eh^i$ for some $e\in E$ and $1\leq i \leq p$.
By Lemma \ref{Zindexp} we have that $Z$ is the preimage in $E$ of $\Z(\N_S(E)/\Phi(E))$ and $\Phi(\N_S(E)) \leq Z$. So $[e,h^i]\in Z$ commutes with $e$ and $h^i$ modulo $\Phi(E)$ and by \cite[Lemma 2.2.2]{GOR} we get
\[ g^p = (eh^i)^p= (h^i)^pe^p[h^i,e]^{\frac{p(p-1)}{2}} \equiv 1 \mod \Phi(E).\]
Hence the group $\N_S(E)/\Phi(E)$ has exponent $p$.

As a consequence, we deduce that
$\Phi(\N_S(E)) = [\N_S(E), \N_S(E)]\Phi(E)$. By Lemma \ref{Zindexp} we have $[\N_S(E)/\Phi(E) \colon \Z(\N_S(E)/\Phi(E))]=p^2$, so $[\Phi(\N_S(E)) \colon \Phi(E)]=p$. Therefore
\[ [\N_S(E) \colon \Phi(\N_S(E))]=[\N_S(E) \colon E][E \colon \Phi(\N_S(E))] = \frac{p[E \colon \Phi(E)]}{[\Phi(\N_S(E)) \colon E]} = [E \colon \Phi(E)].\]
Hence $E$ and $\N_S(E)$ have the same rank.
\end{proof}

We end this section proving that the Frattini subgroup $\Phi(E)$ of $E$ is $\F$-characteristic in $\N_S(E)$.
Note that we have already seen in Lemma \ref{frattini.is.char} that this is true when $E$ has rank $2$ (and in that case we didn't need extra assumptions on the sectional rank of $\N_S(E)$).

\begin{lemma}\label{char.frat}
 The Frattini subgroup $\Phi(E)$ of $E$ is $\F$-characteristic in $\N_S(E)$.
\end{lemma}

\begin{proof}
If $\Phi(E)=T$ then this follows from the definition of $T$. Thus by Theorem \ref{characterization.E}(3) we may assume that $\Phi(E) < T$, $[E \colon T]  =p^2$ and $E$ has rank $3$.
Suppose for a contradiction that there exists $\alpha \in \Aut_\F(\N_S(E))$ such that $\Phi(E)\alpha \neq \Phi(E)$. Since $[T \colon \Phi(E)]=p$ and $T$ is $\F$-characteristic in $\N_S(E)$, we get $T= \Phi(E)\Phi(E)\alpha$. In particular $T\leq \Phi(\N_S(E))$ and since $\N_S(E)/T$ is non-abelian by Lemma \ref{non.abelian.quotient}, we deduce that $T<\Phi(\N_S(E))$. Hence $[\N_S(E) \colon \Phi(\N_S(E))]\leq p^2$ and $\N_S(E)$ has rank $2$, contradicting Lemma \ref{same.rank}.
Therefore $\Phi(E)$ is $\F$-characteristic in $\N_S(E)$.
\end{proof}

\section{Interplay of $\F$-essential subgroups that are $\F$-characteristic in $S$}

Throughout this section, we assume the following hypothesis.

\begin{hp} Suppose that $p$ is an odd prime, $S$ is a $p$-group of sectional rank $3$, $\F$ is a saturated fusion system on $S$ and $E_1$ and $E_2$ are distinct $\F$-essential subgroups of $S$ that are $\F$-characteristic in $S$. Set  $T=\coreF(E_1,E_2)$.
\end{hp}
We now study the interplay of the $\F$-essential subgroups $E_1$ and $E_2$.

\begin{lemma}\label{centr.E.T}
Let $G_i$  be a model for $\N_\F(E_i)$ and let $A_i=\langle S^{G_i}\rangle$. Then \[\C_{A_i}(E_i/T) \leq E_i/T \quad \text{ for every } 1\leq i \leq 2\]
and either
\begin{enumerate}
\item $\C_{G_i}(E_i/T) \leq E_i/T$ for every $1 \leq i \leq 2$; or
\item $\Phi(E_1)=\Phi(E_2) < T$, $S/T \cong p^{1+2}_+$, $E_1$ and $E_2$ have rank $3$ and $A_1/E_1 \cong A_2/E_2 \cong \SL_2(p)$. Moreover there exists a morphism $\theta \in \Out_\F(S)$ of order dividing $p-1$ that centralizes $S/T$  and acts non-trivially on the quotient  $T/\Phi(E_1) \cong \C_{p}$.
\end{enumerate}
\end{lemma}

An example of the situation described in part (2) of Lemma \ref{centr.E.T} is given by the fusion category of the group $G=(\C_p \colon \C_{p-1}) \times \PSL_3(p)$ on one of its Sylow $p$-subgroups $S\cong  \C_p \times p^{1+2}_+$.

\begin{proof}
Clearly if $\C_{G_i}(E_i/T) \leq E_i/T$ then $\C_{A_i}(E_i/T) \leq E_i/T$.

Suppose  $\C_{G_i}(E_i/T) \nleq E_i/T$ for some $1 \leq i\leq 2$.
Set $E=E_i$, $G=G_i$ and $A=A_i$.
Note that $E/T = O_p(G/T)$. Hence if $\C_{G}(E/T)$ is a $p$-group then $\C_{G}(E/T)\leq E/T$, contradicting the assumptions. So there exists a non-trivial element $g\in \C_{G}(E/T)$ such that $(o(g),p)=1$.
If $T\leq \Phi(E)$ then $g$ centralizes $E/\Phi(E)$ and so $g=1$ by Burnside's Theorem (\cite[Theorem 5.1.4]{GOR}), a contradiction. Thus we have $T\nleq \Phi(E)$. Since $S$ has sectional rank $3$ by assumption, we can apply Lemma \ref{T.strict.intersection} and we deduce that $A/E \cong \SL_2(p)$, $[T\Phi(E) \colon \Phi(E)] =p $ and $T\Phi(E)/\Phi(E)=\C_{E/\Phi(E)}(A/E)$. In particular $E$ has rank $3$.
Note that $\C_A(E/T)$ stabilizes the series of subgroups:
\[ \Phi(E) < T\Phi(E) < E.\]
Hence by Lemma \ref{char.series} we get $\C_A(E/T)\leq E/T$.
Note that this proves that \[\C_{A_k}(E_k/T)\leq E_k/T \text{ for every } 1\leq k\leq 2.\]
Also, we get that $g$ acts non-trivially on $T\Phi(E)/\Phi(E)\cong \C_p$ and $g\notin A$.
Hence $g$ has order dividing $p-1$ and since $G=A\N_G(S)$ by the Frattini Argument, we may assume that $g \in \N_G(S)$.

Suppose that $g$ does not centralize $S/T$. Since $[S \colon E]=p$ by Theorem \ref{class.3}, we deduce that $E/T = \C_{S/T}(g)$. Hence by coprime action we get
\[ S/T \cong E/T \times [S/T,g].\]
Thus $[S/T,g]$ is a subgroup of $S/T$ that commutes with $E/T$ and so stabilizes the series $\Phi(E) < T\Phi(E) < E$. Hence by Lemma \ref{char.series} we have $[S,g]T\leq E$, a contradiction.

Thus $g$ centralizes the group $S/T$.
Note that $S/T\Phi(E)$ is a Sylow $p$-subgroup of the group $A/T\Phi(E) \cong (\C_p \times \C_p)\colon \SL_2(p)$.
Hence $S/T\Phi(E)\cong p^{1+2}_+$.
Also, since $g$ centralizes $S/T$ but acts non-trivially on $T\Phi(E)/\Phi(E)$, every element of $T\Phi(E)/\Phi(E)$ is not a $p$-th power of an element in $S/\Phi(E)$. Hence the group $S/\Phi(E)$ has exponent $p$.

Let $P=E_j$ for $j\neq i$ and consider the group $P/\Phi(E)$, that has order $p^3$ and exponent $p$. Since $g$ centralizes $P/T\Phi(E)$ and acts non-trivially on $T\Phi(E)/\Phi(E)$,
we deduce that $P/\Phi(E)$ is elementary abelian.
Since $S$ has sectional rank $3$, we conclude $\Phi(E)=\Phi(P)$. In particular $P$ has rank $3$,  $\Phi(E) \leq T$ (because $\Phi(E)=\Phi(P)$ is $\F$-characteristic in $E$, $P$ and $S$) and $S/T \cong p^{1+2}_+$.
Also, by Lemma \ref{T.strict.intersection} we have $O^{p'}(\Out_\F(P)) \cong \SL_2(p)$.
Finally set $\theta=c_g\Inn(S) \in \Out_\F(S)$.
\end{proof}

\begin{theorem}\label{typeI}
Either $S/T \cong p^{1+2}_+$ or $S/T$ is isomorphic to a Sylow $p$-subgroup of the group $\Sp_4(p)$.
\end{theorem}

\begin{proof}
Suppose $S/T$ is not isomorphic to the group $ p^{1+2}_+$.
Let $G_{12}$ be a model for $\N_\F(S)$ and for every $1\leq i\leq 2$ let $G_i$ be a model for $\N_{\F}(E_i)$ and set $A_i=\langle S^{G_i} \rangle$.
We show that the amalgam $\mathcal{A}(G_1/T, G_2/T, G_{12}/T, A_1/T, A_2/T)$ is a weak $BN$-pair of rank $2$ (as defined in \cite{DGS}).
It is enough to prove the following:
\begin{enumerate}
\item there exist  monomorphisms $\phi_1\colon G_{12} \rightarrow G_1$ and $\phi_2\colon G_{12} \rightarrow G_2$ such that $G_{12}\phi_i = \N_{G_i}(S)$ and $\phi_i|_S = \rm{id}_S$;
\item $A_i/T \cap G_{12}\phi_i/T$ is the normalizer of a Sylow $p$-subgroup of $A_i/T$;
\item $E_i/T \leq A_i/T$ and $G_i/T=(A_i \cdot G_{12}\phi_i)/T$;
\item $\C_{G_i/T}(E_i/T) \leq E_i/T$;
\item $A_i/E_i$ is isomorphic to either $\SL_2(p)$ or $\PSL_2(p)$; and
\item if $H/T$ is a subgroup of  $G_{12}/T$ such  that $H\phi_i/T \norm \G_i/T$ for every $i$ then $H/T =1$.
\end{enumerate}

Note that  the groups $\N_{G_1}(S)$ and $\N_{G_2}(S)$ are models for $\N_\F(S)$. Hence the existence of the monomorphisms $\phi_1$ and $\phi_2$ is guaranteed by the Model Theorem for constrained fusion systems (\cite[Theorem 5.10]{AKO}). In particular $A_i/T \cap  G_{12}\phi_i/T = \N_{A_i}(S)/T$ is the normalizer of the Sylow $p$-subgroup $S/T$ of $A_i/T$.
 Point $(3)$ follows from the Frattini Argument and point $(4)$ is a consequence of Lemma \ref{centr.E.T} and the assumption that $S/T$ is not isomorphic to $p^{1+2}_+$. By Theorem \ref{Op'.typeI} we get point $(5)$. Let $H\leq G_{12}$ be the subgroup described in point $(6)$. Since $\phi_i$ is injective and acts as the identity on $S$ we deduce that $H \cap S = H\phi_i \cap S$ for every $i$.
Note that $H\phi_i \cap S\in \Syl_p(H\phi_i)$ and since $S=O_p(G_{12})$, we deduce that $H\phi_i\cap S \norm H\phi_i$. Thus $H\phi_i\cap S$ is the unique Sylow $p$-subgroup of $H\phi_i$ and is therefore characteristic in $H\phi_i$.
Hence $S \cap H\phi_i \norm G_i$ for every $1 \leq i \leq 2$ and so $S \cap H = S\cap H\phi_1 = S \cap H\phi_2 \leq O_p(G_1) \cap O_p(G_2) = E_1\cap E_2$. By the maximality of $T$ we deduce that $S\cap H \leq T$.
In particular we have
\[ [E_i,H\phi_i] \leq E_i \cap H\phi_i \leq S \cap H\phi_i \leq T.\]
So $H\phi_i/T$ is a subgroup of $G_i/T$ centralizing $E_i/T$ for every $i$.
By Lemma \ref{centr.E.T} (and the assumption that $S/T$ is not isomorphic to $p^{1+2}_+$) we deduce $H\phi_i \leq E_i$. Thus $H = H\phi_i$ and $H\leq E_1 \cap E_2$ is normalized by $G_1$ and $G_2$. By definition of $T$ we then get $H\leq T$ and so $H/T=1$.  Hence point $(6)$ holds.

Therefore $\mathcal{A}(G_1/T, G_2/T, G_{12}/T, A_1/T, A_2/T)$ is a weak $BN$-pair of rank $2$.  In particular the quotient $S/T$ is isomorphic to a Sylow $p$-subgroup of one of the groups listed in \cite[Theorem II.4.A]{DGS}.
Since $p$ is odd and $S/T$ has sectional rank at most $3$, by \cite[Theorem 3.3.3]{GLS3} we deduce that $S/T$ is isomorphic to a Sylow $p$-subgroup of either $\PSL_3(p)$ or $\PSp_4(p)$. Finally notice that the Sylow $p$-subgroups of $\PSL_3(p)$ are isomorphic to the group $p^{1+2}_+$ and that the Sylow $p$-subgroups of $\PSp_4(p)$ are isomorphic to the Sylow $p$-subgroups of $\Sp_4(p)$.
\end{proof}

\begin{lemma}\label{PSLtype}
If $S/T \cong p^{1+2}_+$ then $E_1$ and $E_2$ are abelian,
$E_1 \cap E_2=\Z(S)$ and for every $1\leq i \leq 2$ the group $T$ is the centralizer in $\Z(S)$ of $O^{p'}(\Aut_\F(E_i))$.
\end{lemma}

\begin{proof}
Note that $E_1/T \cong E_2/T \cong \C_p \times \C_p$.
Thus $\Phi(E_i) \leq T$ and by Theorem \ref{class.3} we have $O^{p'}(\Out_\F(E_i))\cong \SL_2(p)$. In particular $E_i/T$ is a natural $\SL_2(p)$-module for $O^{p'}(\Out_\F(E_i))$.
By  Lemma \ref{properties.T} we have $E_i=\C_{E_i}(T)T$ and $\C_S(T)\nleq E_1 \cap E_2$. Thus we may assume $\C_S(T) \nleq E_1$ and so $O^{p'}(\Out_\F(E_1))$ centralizes $T$ (again by Lemma  \ref{properties.T}). Therefore by Theorem \ref{char.T} we deduce that $T$ is abelian, $T\leq \Z(E_1)$, $|[E_1,E_1]|\leq p$ and $T/[E_1,E_1]$ is cyclic.
The fact that $T$ is abelian implies that $E_2=\C_{E_2}(T)T = \C_{E_2}(T)$ and so $T\leq \Z(E_2)$. Since $S=E_1E_2$ we conclude that $T\leq \Z(S)$.

Since $E_1$ is receptive, every morphism in $\N_{O^{p'}(\Aut_F(E_1))}(\Aut_S(E_1))$ is the restriction of an $\F$-automorphism of $\N_S(E_1)=S$. Hence there exists a morphism $\tau \in \Aut_\F(S)$ that acts on $E_1/T$ as the involution $\begin{pmatrix} -1 & 0 \\ 0 & -1 \end{pmatrix}$, with respect to the basis $\{ e_1T, xT\}$ , for some $e_1\in E_1 \backslash E_2$ and $x\in E_1 \cap E_2$.
Since $S/T\cong p^{1+2}_+$, we have $[E_1,E_2]T=E_1\cap E_2$. 
The group $E_2$ is $\F$-characteristic in $S$, so $\tau$ acts on $E_2$ and centralizes the quotient $E_2/(E_1\cap E_2) \cong S/E_1$.
Let $y\in E_2 \backslash E_1$. Then $x\tau =x^{-1}t_1$ and $y\tau=yt_2$, for some $t_1,t_2\in T$. Since $T\leq \Z(S)$ and $\tau$ centralizes $T$, we have
\[ [x,y] = [x,y]\tau= [x^{-1}t_1, yt_2]= [x,y]^{-1}.\]
Since $p$ is an odd prime, we deduce that $[x,y]=1$ and the group $E_2$ is abelian.
Note that $\C_S(T)=S\nleq E_2$ so  $O^{p'}(\Aut_\F(E_2))$ centralizes $T$ by Lemma  \ref{properties.T}. Hence we can repeat the same argument with $E_2$ in place of $E_1$ to prove that $E_1$ is abelian.

Since $E_1$ and $E_2$ are abelian and $S=E_1E_2$, we deduce $E_1 \cap E_2=\Z(S)$. Also, since there exists an involution in $O^{p'}(\Aut_\F(E_i))$ that inverts the quotient $\Z(S)/T$, we conclude that  for every $1\leq i \leq 2$ the group $T$ is the centralizer in $\Z(S)$ of $O^{p'}(\Aut_\F(E_i))$.

\end{proof}

\begin{lemma}\label{PSptype}
 If $S/T$ is isomorphic to a Sylow $p$-subgroup of $\Sp_4(p)$ then, up to interchanging the definitions of $E_1$ and $E_2$, the following hold:
\begin{enumerate}
\item  $\Z(S)=\Z(E_1)$ is the preimage in $S$ of $\Z(S/T)$;
\item $E_1/T \cong p^{1+2}_+$ and $O^{p'}(\Out_\F(E_1))\cong \SL_2(p)$;
\item $E_2$ is abelian, $T=\Phi(E_2)$ and  $O^{p'}(\Out_\F(E_2))\cong \PSL_2(p)$.
\end{enumerate}
\end{lemma}

\begin{proof}
Note that $S/T$ has order $p^4$, center of order $p$ and a unique elementary abelian maximal subgroup; every other maximal subgroup of $S/T$ is extraspecial.  Thus we may assume that $E_1/T$ is extraspecial.
Note that $T\nleq \Phi(E_1)$ and since $S$ has sectional rank $3$,  by Lemma \ref{T.strict.intersection} we get $O^{p'}(\Out_\F(E_1)) \cong \SL_2(p)$ and $[T\Phi(E_1) \colon T] =p$. Note that $T\Phi(E_1)$ is normal in $S$, so $T\Phi(E_1)/T=\Z(S/T)$. If $E_1/T$ has exponent $p^2$ and $\Omega$ is the preimage in $E_1$ of $\Omega_1(E/T)$, then $[E \colon \Omega] =p$ and $\Aut_S(E_1)$ stabilizes the series $\Phi(E_1) < T\Phi(E_1) < \Omega < E_1$. So $\Aut_S(E_1) \leq O_p(\Aut_\F(E_1))=\Inn(E_1)$ by Lemma \ref{char.series}, a contradiction. Thus $E_1/T$ has exponent $p$ and so  $E_1/T\cong p^{1+2}_+$.

If $T\nleq \Phi(E_2)$ then by Lemma \ref{T.strict.intersection} we have $[T\Phi(E_2) \colon T]=p$. Thus $T\Phi(E_2)/T=\Z(S/T)=T\Phi(E_1)/T$, contradicting the maximality of $T$. Therefore $T\leq \Phi(E_2)$ and $\Phi(E_2)/T < \Z(S/T)$. Since $S$ has sectional rank $3$ and $|\Z(S/T)|=p$ we deduce that $T=\Phi(E_2)$.

Suppose that the group $O^{p'}(\Out_\F(E_2))$ is isomorphic to $\SL_2(p)$  and let $C\leq E_2$ be the preimage in $E_2$ of  the group $\C_{E_2/T}(O^{p'}(\Aut_\F(E_2)))$. Then $[C \colon T]  =p$ and since $C\norm S$ we get $C/T=\Z(S/T)=T\Phi(E_1)/T$. Hence $C$ is $\F$-characteristic in $E_1$, $E_2$ and $S$, contradicting the maximality of $T$. Therefore $O^{p'}(\Out_\F(E_2))$ is not isomorphic to $\SL_2(p)$. Since $S$ has sectional rank $3$ we can apply Theorem \ref{Op'.typeI}  to deduce that $O^{p'}(\Out_\F(E_2))\cong \PSL_2(p)$.
In particular the group $\Out_\F(E_2)$ acts irreducibly on $E_2/T$ and since $\C_{E_2}(T)\nleq T$ by Lemma \ref{properties.T} and $T\C_{E_2}(T)$ is $\F$-characteristic in $E_2$, we conclude $T\C_{E_2}(T)=E_2$.

If $\C_S(T) \leq E_1$ then $E_2\leq T\C_S(T)\leq E_1$, a contradiction.
Thus $\C_S(T)\nleq E_1$ and $O^{p'}(\Out_\F(E_1))$ centralizes $T$ by Lemma \ref{properties.T}.
Also, $E_1\cap E_2 \leq E_2 \leq T\C_S(T)$. So $E_1 \cap E_2 \leq T\C_{E_1}(T)$. Since $E_1$ is $\F$-essential, by Lemma \ref{char.series} no proper non-trivial subgroup of $E_1/T\Phi(E_1)$ can be $\F$-characteristic in $E_1$. Therefore we conclude that  $E_1=T\C_{E_1}(T)$ and $S=T\C_S(T)$.

Note that $E_1/\Z(T) \cong T/\Z(T) \times \C_{E_1}(T)/\Z(T)$ and $\C_{E_1}(T)/\Z(T) \cong E_1/T \cong p^{1+2}_+$.
Since $S$ has sectional rank $3$, we deduce that the group $T/\Z(T)$ is cyclic and so $T$ is abelian.
Hence $S=\C_S(T)$ and $T\leq \Z(S)$.

The quotient $E_1/T\Phi(E_1)$ is a natural $\SL_2(p)$-module for the group $O^{p'}(\Out_\F(E_1))$ and the group $E_1$ is receptive. Hence there exists a morphism $\tau \in \Aut_\F(S)$ that acts on $E_1/T\Phi(E_1)$ as the involution $\begin{pmatrix} -1 & 0 \\ 0 & -1 \end{pmatrix}$, with respect to the basis $\{e_1T\Phi(E_1), xT\Phi(E_1)\}$, for some $e_1 \in E_1 \backslash E_2$ and $x \in E_1 \cap E_2$. Also, $\tau$ centralizes $T\Phi(E_1)/T$ and since $[E_1,E_2]T=E_1\cap E_2$, the morphism $\tau$ centralizes the quotient $E_2/(E_1 \cap E_2)$.

Since $\langle x \rangle T\Phi(E_1)/ T\Phi(E_1)$ is the only section of $E_2$ that is not centralized by $\tau$ and $T\leq \Z(E_2)$, we deduce that $\langle x \rangle T\leq \Z(E_2)$. Since $\Z(E_2)$ is $\F$-characteristic in $E_2$ and $\Out_\F(E_2)$ acts irreducibly on $E_2/T$, the only possibility is $E_2=\Z(E_2)$. Thus $E_2$ is abelian.
Similarly, since $T\leq \Z(S)$ and $T\Phi(E_1)/T$ is the only section of $E_1/T$ not inverted by $\tau$, we deduce that $T\Phi(E_1)\leq \Z(E_1)$. Since the group $E_1$ is non-abelian ($E_1/T\cong p^{1+2}_+$) we conclude that $T\Phi(E_1)=\Z(E_1)=\Z(S)$.
\end{proof}

We end this section proving that if $p$ is an odd prime, $S$ has sectional rank $3$,  there are two $\F$-essential subgroups of $S$ that are $\F$-characteristic in $S$ and $O_p(\F)=1$ then $S$ is isomorphic to a Sylow $p$-subgroup of the group $\Sp_4(p)$.

\begin{theorem}\label{Sp4}
Suppose $p$ is an odd prime,  $S$ is a $p$-group of sectional rank $3$ and $\F$ is a saturated fusion system on $S$ such that $O_p(\F)=1$. Assume there exist distinct $\F$-essential subgroups $E_1$ and $E_2$ of $S$ both $\F$-characteristic in $S$. Then $S$ is isomorphic to a Sylow $p$-subgroup of the group $\Sp_4(p)$, $E_i \cong p^{1+2}_+$ and $E_j\cong \C_p \times \C_p \times \C_p$, where $\{i,j\} = \{1,2\}$.
\end{theorem}

\begin{proof} Set $T=\coreF(E_1,E_2)$. We aim to prove that $T \norm \F$.
By Theorem \ref{typeI} either $S/T \cong p^{1+2}_+$ or $S/T$ is isomorphic to a Sylow $p$-subgroup of the group $\Sp_4(p)$. Note that the group $p^{1+2}_+$ has sectional rank $2$. Since $S$ has sectional rank $3$, if $T\norm \F$ then $T=O_p(\F)=1$ and $S$ is isomorphic to a Sylow $p$-subgroup of the group $\Sp_4(p)$. Note that by Lemmas \ref{PSLtype} and \ref{PSptype} we have $T\leq \Z(S)$. So $T$ is contained in every $\F$-essential subgroup of $S$ (recall that $\F$-essential subgroups are self-centralizing in $S$). By \cite[Proposition I.4.5]{AKO} and the fact that $T$ is $\F$-characteristic in $S$ by definition, to prove that $T$ is normal in $\F$ it is enough to show that $T$ is $\F$-characteristic in every $\F$-essential subgroup of $S$. By definition $T$ is $\F$-characteristic in $E_1$ and $E_2$. Suppose there exists an $\F$-essential subgroup $E_3$ of $S$ distinct from $E_1$ and $E_2$.

Assume $S/T\cong p^{1+2}_+$.
Since $\Z(S)<E_3$ and $\Z(S)$ has index $p^2$ in $S$ we get that $E_3$ is abelian and normal in $S$.
\begin{itemize}
\item Suppose $E_3$ is $\F$-characteristic in $S$ and set $T_{1,3}= \coreF(E_1,E_3)$. Since both $E_1$ and $E_3$ are abelian, by Theorem \ref{typeI} and Lemma \ref{PSptype} we deduce that $S/T_{1,3}\cong p^{1+2}_+$. Hence by Lemma  \ref{PSLtype} we have $T_{1,3} = \C_{\Z(S)}(O^{p'}(\Out_\F(E_1))) = T$. Therefore $T$ is $\F$-characteristic in $E_3$.

\item Suppose $E_3$ is not $\F$-characteristic in $S$. Then $O^{p'}(\Out_\F(E_3))\cong \SL_2(p)$ by Theorem \ref{typeIIauto} and since $E_3$ is receptive there exists an $\F$-automorphism $\varphi$ of $\N_S(E_3)=S$ that inverts $E_3/C$, where $C$ is the preimage in $E_3$ of the group $\C_{E_3/\Phi(E_3)}(O^{p'}(\Out_\F(E_3))$. Thus $\varphi$ inverts $E_3/\Z(S)$ and centralizes $S/E_3$. In particular the action of $\varphi$ on $S/\Z(S)$ is not scalar. However, $\varphi$ normalizes $E_1$, $E_2$ and $E_3$ and we get a contradiction.
\end{itemize}
Hence if $S/T \cong p^{1+2}_+$ then the group $T$ is $\F$-characteristic in every $\F$-essential subgroup of
$S$ and is therefore normal in $\F$.

 Assume $S/T$ is isomorphic to a Sylow $p$-subgroup of the group $\Sp_4(p)$. Then by Lemma \ref{PSptype} we can assume that $E_1/T\cong p^{1+2}_+$ and $E_2$ is abelian. Note that $[S \colon \Z(S)]=p^3$ so $[E_3 \colon \Z(S)]\leq p^2$.

Suppose $[E_3 \colon \Z(S)]=p^2$. Then $E_3$ is normal in $S$.
If $\Z(S)$ is not $\F$-characteristic in $E_3$, then $\Z(S) < \Z(E_3)$ and so $E_3$ is abelian. In particular $\Z(S)=E_2 \cap E_3$ has index $p$ in $E_3$, which is a contradiction.
Thus $\Z(S)$ is $\F$-characteristic in $E_3$. Let $G_3$ be a model for $\N_\F(E_3)$. Then  $T^g\leq \Z(S)$ for every $g\in G_3$. Since $T$ is $\F$-characteristic in $S$ and $G_3=\langle S^{G_3} \rangle\N_{G_3}(S)$ by the Frattini Argument, we conclude that $G_3$ normalizes $T$ and so $T$ is $\F$-characteristic in $E_3$.

Suppose $[E_3\colon \Z(S)]=p$. Then $E_3$ is abelian and not normal in $S$. Set $T_3=\coreF(E_3)$ and suppose $T\neq T_3$.
Note that $\Z(S)/T_3=\Z(S/T_3)$ (since $\Z(S/T_3) < E/T_3$) and so $[ E_3/T_3 \colon \Z(S/T_3)]=p$. Thus by Lemma \ref{centerN2} we have $[E_3 \colon T_3] =p^2$.

Let $\varphi \in O^{p'}(\Aut_\F(E_3))$ be a morphism that normalizes $\Z(S)$, inverts the quotient $E_3/T_3$ and centralizes $T_3$. Such a morphism exists because $O^{p'}(\Out_\F(E_3))\cong \SL_2(p)$ by Theorem \ref{typeIIauto} and $T_3$ is centralized by  $O^{p'}(\Out_\F(E_3))$ (Lemma \ref{prop.T.II}).
Note that $\varphi$ is a restriction to $E_3$ of an $\F$-automorphism $\ov{\varphi}$ of $S$ by Lemma \ref{lift.to.N2}. Also we have $[E_3,S] \leq \Z_2(S) \backslash \Z(S)$ and $[E_3, \Z_2(S)]\leq \Z(S) \backslash T_3$.  Using properties of commutators (\cite[Theorem 2.2.1, Lemma 2.2.2]{GOR}), it is not hard to see that the action of $\ov{\varphi}$ on the sections of $S/T\cap T_3$ is as described in Figure \ref{PSp}.

\begin{figure}[H]
\centering
\begin{tikzpicture}[x=1.00mm, y=1.00mm, inner xsep=0pt, inner ysep=0pt, outer xsep=0pt, outer ysep=0pt]
%\path[line width=0mm] (-17.42,-46.43) rectangle +(34.30,62.56);
\definecolor{L}{rgb}{0,0,0}
\path[line width=0.30mm, draw=L] (-0.0,-10.) -- (-10.0,0.00);
\definecolor{F}{rgb}{0,0,0}
\path[line width=0.15mm, draw=L, fill=F] (-10.,-0.) circle (0.50mm);
\path[line width=0.15mm, draw=L, fill=F] (-0.0,10) circle (0.50mm);
\path[line width=0.15mm, draw=L, fill=F] (-0.0,-10.0) circle (0.50mm);
\path[line width=0.15mm, draw=L, fill=F] (10.0,-0.0) circle (0.50mm);
\draw(11.3,-31) node[anchor=base west]{\fontsize{8.54}{10.24}\selectfont $T$};
\draw(2.5,-11.53) node[anchor=base west]{\fontsize{8.54}{10.24}\selectfont $\Z_2(S)$};
\draw(11.3,-1.12) node[anchor=base west]{\fontsize{8.54}{10.24}\selectfont $E_1$};
\draw(-15,-10.66) node[anchor=base west]{\fontsize{8.54}{10.24}\selectfont $E_3$};
\path[line width=0.15mm, draw=L, fill=F] (0.0,-20.0) circle (0.50mm);
\draw(-0.82,11.78) node[anchor=base west]{\fontsize{8.54}{10.24}\selectfont $S$};
\path[line width=0.15mm, draw=L, fill=F] (-10,-10.) circle (0.50mm);
\draw(2.5,-20.70) node[anchor=base west]{\fontsize{8.54}{10.24}\selectfont $\Z(S)$};
\path[line width=0.30mm, draw=L] (10,0.) -- (-0.,10.);
\path[line width=0.30mm, draw=L] (-10,-10.) -- (-10,-0.0);
\path[line width=0.30mm, draw=L] (-10,-0.0) -- (0.,10.0);
\path[line width=0.30mm, draw=L] (10,-30) -- (-10.,-10);
\path[line width=0.30mm, draw=L] (-0.,-10.) -- (-0.,-20.);
\path[line width=0.30mm, draw=L] (-10,-30) -- (0,-20);
\path[line width=0.30mm, draw=L] (-0.0,-40.00) -- (-10.0,-30.0);
\path[line width=0.30mm, draw=L] (-0.0,-40) -- (10.0,-30);
\path[line width=0.15mm, draw=L, fill=F] (-10,-30) circle (0.50mm);
\path[line width=0.15mm, draw=L, fill=F] (10.0,-30) circle (0.50mm);
\path[line width=0.15mm, draw=L, fill=F] (-0.0,-40.0) circle (0.50mm);
\draw(-15,-31) node[anchor=base west]{\fontsize{8.54}{10.24}\selectfont $T_3$};
\draw(-20.94,-0.55) node[anchor=base west]{\fontsize{8.54}{10.24}\selectfont $\N_S(E_3)$};
\draw(-5,-43.78) node[anchor=base west]{\fontsize{8.54}{10.24}\selectfont $T_3 \cap T$};
\definecolor{T}{rgb}{1,0,0}
\draw[T] (-7.52,-37.05) node[anchor=base west]{\fontsize{8.54}{10.24}\selectfont \textcolor[rgb]{0, 0, 0}{1}};
\draw[T] (-7.95,-25) node[anchor=base west]{\fontsize{8.54}{10.24}\selectfont \textcolor[rgb]{0, 0, 0}{-1}};
\draw[T] (-7.95,-16.90) node[anchor=base west]{\fontsize{8.54}{10.24}\selectfont \textcolor[rgb]{0, 0, 0}{-1}};
\draw[T] (1.11,-16) node[anchor=base west]{\fontsize{8.54}{10.24}\selectfont \textcolor[rgb]{0, 0, 0}{1}};
\draw[T] (-7.38,6.28) node[anchor=base west]{\fontsize{8.54}{10.24}\selectfont \textcolor[rgb]{0, 0, 0}{-1}};
\draw[T] (6.86,-6.91) node[anchor=base west]{\fontsize{8.54}{10.24}\selectfont \textcolor[rgb]{0, 0, 0}{-1}};
\draw[T] (-13,-6.82) node[anchor=base west]{\fontsize{8.54}{10.24}\selectfont \textcolor[rgb]{0, 0, 0}{1}};
\draw[T] (5.86,-25) node[anchor=base west]{\fontsize{8.54}{10.24}\selectfont \textcolor[rgb]{0, 0, 0}{1}};
\draw[T] (6.01,-37.05) node[anchor=base west]{\fontsize{8.54}{10.24}\selectfont \textcolor[rgb]{0, 0, 0}{-1}};
\path[line width=0.30mm, draw=L] (0.,-10) -- (10.,0.);
\end{tikzpicture}%
\caption{}\label{PSp}
\end{figure}

In particular we get $[E_1,\Z_2(S)]\leq T$ and so the group $E_1/T$ is abelian, contradicting the assumption that $E_1/T$ is extraspecial. Thus $T=T_3$ is $\F$-characteristic in $E_3$.

Hence $T$ is $\F$-characteristic in every $\F$-essential subgroup of $S$ and is therefore normal in $\F$. This, together with the assumptions that $S$ has sectional rank $3$ and $O_p(\F)=1$, completes the proof.
\end{proof}

\section{Proof of Theorem \ref{normal}}

Throughout this section, we assume the following hypothesis.

\begin{hp}
Let $p$ be an odd prime, let $S$ be a $p$-group of sectional rank $3$ and let $\F$ be a saturated fusion  system on $S$.
Let $E\leq S$ be an $\F$-essential subgroup of $S$ not $\F$-characteristic in $S$ and set $T=\coreF(E) < E$.
\end{hp}

\begin{definition}
For every subgroup $P\leq S$ containing $\Z(S)$ we define
\[Z_P=\langle ~ \Omega_1(\Z(S)) ^ {\Aut_\F(P)} ~ \rangle.\]
\end{definition}

\begin{remark}
Note that $Z_S=\Omega_1(\Z(S))$ and $Z_S \leq Z_P \leq \Omega_1(\Z(P))$.
In particular $Z_P$ is elementary abelian and the assumption on the sectional rank of $S$ implies $|Z_P|\leq p^3$.
\end{remark}

The group $Z_S$ is an $\F$-characteristic subgroup of $S$ contained in every $\F$-essential subgroup of $S$ (recall that every $\F$-essential subgroup is self-centralizing in $S$). If $O_p(\F)=1$ then there exists and $\F$-essential subgroup $P$ of $S$ such that $Z_S < Z_P$. When this happens we say that $P$ \emph{moves} $Z_S$. In this section we study the structure of the $\F$-essential subgroups of $S$ that move $Z_S$, aiming for the proof of Theorem \ref{normal}.

\begin{lemma}\label{lemmaZS}
If $Z_ST = Z_ET$ then $Z_S \leq T$.
\end{lemma}

\begin{proof}
Recall that by definition $T$ is the largest subgroup of $E$ that is $\F$-characteristic in both $E$ and $\N_S(E)$.
Note that the group $Z_ET$ is $\F$-characteristic in $E$ and the group $Z_ST$ is normalized by $\Aut_S(\N_S(E))\cong \N_S(\N_S(E))/\Z(\N_S(E))$.
Suppose  $Z_ST=Z_ET$. If $\N_S(E)=S$ then $Z_ST$ is $\F$-characteristic in $\N_S(E)$. If $\N_S(E) < S$ then $\Aut_\F(\N_S(E))=\Aut_S(\N_S(E))\N_{\Aut_\F(\N_S(E))}(E)$ by Lemma \ref{lift.to.N2} and so $Z_ST=Z_ET$ is $\F$-characteristic in $\N_S(E)$. Therefore in any case the group $Z_ST=Z_ET$ is $\F$-characteristic in $E$ and $\N_S(E)$. Hence $Z_ST = T$ by maximality of $T$ and so $Z_S \leq T$.
\end{proof}

\begin{theorem}\label{center.fixed.in.T}
We have
\[ Z_S = Z_E \text{ if and only if } Z_S \leq T.\]
\end{theorem}

\begin{proof}
If $Z_S = Z_E$  then $Z_S \leq T$ by Lemma \ref{lemmaZS}. We want to prove that if $Z_S < Z_E$ then $Z_S \nleq T$.
Aiming for a contradiction, assume there exists an $\F$-essential subgroup
$E$ of $S$, not $\F$-characteristic in $S$, such that $Z_S < Z_E$ and $Z_S \leq T$. We can choose $E$ to be a maximal counterexample (with respect to inclusion) among the $\F$-essential subgroups of $S$ not $\F$-characteristic in $S$.
From $Z_S \leq T$ we get $Z_E \leq T$. So $Z_E\leq \Omega_1(T)$ and by Theorem \ref{characterization.E} and the fact that $Z_S < Z_E$ we conclude $|Z_E|=p^2$ and $|Z_S|=p$.
By Theorem \ref{prop.T.II} the group $O^{p'}(\Out_\F(E))$ centralizes $T$.
Note that $\Inn(S)$ acts trivially on $Z_S$, so the group $O^{p'}(\Aut_\F(E))$ centralizes $Z_S$.
By the Frattini argument we have
\[ \Aut_\F(E) = O^{p'}(\Aut_\F(E))\N_{\Aut_\F(E)}(\Aut_S(E)).\]
Then we may assume that there exists $\alpha \in \N_{\Aut_\F(E)}(\Aut_S(E))$ of order prime to $p$ such that $Z_S\alpha \neq Z_S$. Note that $\alpha$ is the restriction to $E$ of an $\F$-automorphism $\ov{\alpha}$ of $\N_S(E)$ (since $E$ is receptive) but it is not a restriction of an $\F$-automorphism of $S$ (otherwise it normalizes $Z_S$).
 In particular the group $E$ is not abelian by Corollary \ref{lift.Ni.ab}.

By Alperin's Fusion Theorem there exist subgroups $P_1, P_2, \dots , P_n$ of $S$ and morphisms $\phi_i \in \Aut_\F(P_i)$ for every $1 \leq i \leq n$ such that
\begin{itemize}
\item every $P_i$ is either $\F$-essential or equal to $S$,
\item $\N_S(E) \leq P_1 \cap P_n$, and
\item $\phi_1\cdot\phi_2\cdots\phi_n|_{\N_S(E)} = \ov{\alpha}$.
\end{itemize}
Suppose $Z_S\phi_1 = Z_S$. Note that $E\phi_1$ is an $\F$-essential subgroup of $S$ isomorphic to $E$ by Lemma \ref{indexp} and $\Aut_\F(E)=\phi_1\Aut_\F(E\phi_1)\phi_1^{-1}$. In particular $Z_S$ is not normalized by $\Aut_\F(E\phi_1)$. Also $\Z_S \leq T\phi_1=\coreF(E\phi_1)$ by Lemma \ref{same.T} and we can replace $E$ by $E\phi_1$. Thus we can assume that $Z_S\phi_1 \neq Z_S$. In particular $P=P_1$ is an  $\F$-essential subgroup of $S$ containing $\N_S(E)$ such that $Z_S < Z_{P}$.

Suppose $P$ is not $\F$-characteristic in $S$
and set $T_P=\coreF(P)$. Then by the choice of $E$ we have $Z_S \nleq T_P$. In particular $T\nleq T_P$ and since $\Phi(E)\leq \Phi(P)\leq T_P$ by Theorem \ref{characterization.E}(3), we deduce that $[T \colon \Phi(E)]=p$ and $[E \colon T]=p^2$.
In particular $[T \colon T \cap T_P]=[T \colon \Phi(E)]=p$.
Since $T\leq P$ and $T_P\norm P$ we can consider the group $TT_P$ and we get $[TT_P \colon T_P] = [T \colon T \cap T_P]=p$. Since $Z_S \leq T$ and $Z_S \nleq T_P$ we deduce $T_P < Z_ST_P =TT_P$.
Since $E$ is not abelian and $T\leq \Z(E)$ (by Theorem \ref{characterization.E}), we conclude that $T=\Z(E)$. In particular $Z_P \leq \Z(P) \leq \Z(E)=T$ and so $Z_ST_P = Z_PT_P = TT_P$.
Hence $Z_P\leq T_P$ by Lemma \ref{lemmaZS}, that is a contradiction.

We deduce that the $\F$-essential subgroup $P$ has to be $\F$-characteristic in $S$. In particular $Z_P$ is $\F$-characteristic in $S$.
If $Z_P\leq T$ then $Z_P=\Omega_1(T)$ (since $Z_S<Z_P$ and $|\Omega_1(T)|=p^2$). So $[E,E] \leq Z_P$ and by Lemma \ref{lift.Ni} applied with $K=Z_P$ and $\N^j=S$ we conclude that $E$ has maximal normalizer tower in $S$, $P$ is the unique maximal subgroup of $S$ containing $E$ and $P$ is not $\F$-essential, a contradiction.
Thus $Z_P \nleq T$. In particular, since $Z_P \leq \Omega_1(\Z(P)) \leq \Omega_1(\Z(E))$, we get  $\Omega_1(\Z(E))\nleq T$ and so $Z_E<\Omega_1(\Z(E))$.
Since $|Z_E|=p^2$ and $S$ has sectional rank $3$, this implies $|\Omega_1(\Z(E))|=p^3$ and \[ [T\Omega_1(\Z(E)) \colon T] = [\Omega_1(\Z(E)) \colon T \cap\Omega_1(\Z(E))] = [ \Omega_1(\Z(E)) \colon Z_E]=p.\]
Recall that by Theorem \ref{characterization.E} either $E=\C_E(T)$ or $[\C_E(T) \colon T] =p^2$. Since $T\Omega_1(\Z(E)) < \C_E(T)$ and it is $\F$-characteristic in $E$, we deduce that $T\leq \Z(E)$. Also $T\neq \Z(E)$ (otherwise $\Omega_1(\Z(E)) \leq T$) and  $E$ is not abelian, so $[E \colon T]=p^3$ and $[E \colon \Z(E)] = p^2$. Note that $\N_S(E)=E\C_E(T)$ by Lemma \ref{prop.T.II}, so  $T\leq \Z(\N_S(E))$. Also, $\Z(\N_S(E)) < \Z(E)$ otherwise $\Z(\N_S(E))=\Z(E)$ is $\F$-characteristic in $E$ and $\N_S(E)$ and so $\Z(E)$ is contained in $T$, a contradiction. Hence we conclude $T=\Z(\N_S(E))$.
In particular $Z_P \leq \Z(P) \leq \Z(\N_S(E))= T$, and we get a contradiction.

Therefore if $Z_S <Z_E$ then $Z_S \nleq T$.

\end{proof}

\begin{theorem}\label{E.moving.center.is.abelian}
If $Z_S < Z_E$ then $E$ is abelian.
\end{theorem}

\begin{proof}
If $T=1$ then $E$ is elementary abelian by Theorem \ref{characterization.E}(3), so we can assume $T\neq 1$.
By Theorem \ref{center.fixed.in.T} we have $Z_S\nleq T$. So $Z_E\nleq T$ and by Lemma \ref{lemmaZS} we have $T < Z_ST < Z_ET$. Thus $[Z_ET \colon T] \geq p^2$.

Aiming for a contradiction, suppose $T\nleq \Z(E)$ (so $[E \colon T]=p^3$ by Theorem \ref{characterization.E} and $T=\Phi(E)$). Hence $Z_ET=\C_E(T) =\Omega_1(\Z(E))T$ and $[Z_ET \colon T]=p^2$.
In particular, since $S$ has sectional rank $3$ and $T\cap \Omega_1(\Z(E)) \neq 1$, we deduce that $|\Omega_1(\Z(E))|=p^3$, $\Omega_1(\Z(E))=\Omega_1(E)$ and $|\Omega_1(T)|=p$, so $T$ is cyclic.

Let $y\in E$ be of minimal order such that $E = \langle y \rangle \Omega_1(E)T$. We want to show that $y$ commutes with $T$, contradicting the fact that $T \nleq \Z(E)$. Note that $y^p \in \Phi(E)= T$.
Suppose that $\langle y \rangle T$ has rank $2$. Then there exists a normal subgroup of $\langle y \rangle T$ isomorphic to the group $C_p \times C_p$. In particular $y$ has order $p$ and so $y\in \Omega_1(E)$. Thus $E=\Omega_1(E)T=\C_E(T)$ contradicting the assumptions. Thus the group $\langle y \rangle T$ has to be cyclic. In particular $y$ commutes with $T$ and so $T \leq \Z(E)$,  a contradiction.

Therefore $T\leq \Z(E)$. So $Z_ET \leq \Z(E)$. Since $[E \colon Z_ET] \leq p$ we conclude that $E$ is abelian.
\end{proof}

\begin{lemma}\label{char.C}
Suppose $E$ is normal in $S$ and has rank $3$.
Let $C$ be the preimage in $E$ of the group $\C_{E/\Phi(E)}(O^{p'}(\Aut_\F(E)))$.
Then $S/C \cong p^{1+2}_+$, $C/\Phi(E) < \Z(S/\Phi(E))$ and $C=T$.
\end{lemma}

\begin{proof}
Since $E$ is normal in $S$ and has rank $3$, we have $[S \colon E] = p$ by Theorem \ref{class.3} and $O^{p'}(\Out_\F(E)) \cong \SL_2(p)$ by Theorem \ref{typeIIauto}.
Let $\alpha \in \Aut_\F(S)$ be such that $E\alpha \neq E$. By Lemma \ref{char.frat} the group $\Phi(E)$ is $\F$-characteristic in $S$, so $\Phi(E)=\Phi(E\alpha)$. In particular the group $E\alpha/\Phi(E)$ is abelian and if $Z$ is the preimage in $S$ of $\Z(S/\Phi(E)$ then $Z= E \cap E\alpha$ and $[S \colon Z]=p^2$.
By Lemma \ref{same.rank} the group $S/\Phi(E)$ has exponent $p$ and $S$ has rank $3$.
Note that $|C/\Phi(E)| = p$ and $C < Z$.  Also, $C\neq \Phi(S)$ otherwise $\Aut_S(E)$ stabilizes the series $\Phi(E) < C < E$ and so it is normal in $\Aut_\F(E)$ by Lemma \ref{char.series}, contradicting the fact that $E$ is $\F$-essential.
Since $S/C$ has order $p^3$ and exponent $p$ (it is a section of $S/\Phi(E)$ that has exponent $p$), we deduce that $S/C \cong p^{1+2}_+$.

Let $\tau$ be the $\F$-automorphism of $E$ that inverts $E/C$. Then $\tau$ centralizes $C/\Phi(E)$ by Lemma \ref{prop.T.II} and inverts $Z/C$. In particular $\tau$ does not act as a scalar on $Z/\Phi(E)$ and so $C/\Phi(E)$ and $\Phi(S)/\Phi(E)$ are the only maximal subgroups of $Z/\Phi(E)$ normalized by $\tau$. Since $E$ is receptive, $\tau$ is the restriction of an $\F$-automorphism $\ov{\tau}$ of $S$.
Thus $C/\Phi(E)$ and $\Phi(S)/\Phi(E)$ are the only maximal subgroups of $Z/\Phi(E)$ that can be $\F$-characteristic in $S$.
Since the inner automorphisms of $S$ act trivially on $Z/\Phi(E)$ and $S$ is fully automized, the group $\Aut_\F(S)/ \C_{\Aut_\F(S)}(Z/\Phi(E))$ has order prime to $p$. Since $\Phi(S)$ is $\F$-characteristic in $S$, by Maschke's Theorem (\cite[Theorem 3.3.2]{GOR}) there exists a maximal subgroup of $Z/\Phi(E)$ distinct from $\Phi(S)/\Phi(E)$ that is $\F$-characteristic in $S$. Hence $C$ and $\Phi(S)$ are the only maximal subgroups of $Z$ containing $\Phi(E)$ that are $\F$-characteristic in $S$. In particular  $C \leq T$ and since $[E \colon T] \geq p^2$ by Theorem \ref{characterization.E}(3), we deduce that $C=T$.
\end{proof}

\begin{lemma}\label{OpF}
If $E$ is normal and abelian then
$O_p(\F)\neq 1$.
\end{lemma}

\begin{proof}\label{proofC.norm.F}
Since $E$ is normal in $S$ we have $[S \colon E] = p$ by Theorem  \ref{class.3}.
If $E$ has rank $2$ then $E$ is an $\F$-pearl by Theorem \ref{rank2pearl}. So $E\cong \C_p \times \C_p$ and $|S|=p^3$, contradicting the fact that $S$ has sectional rank $3$.
Therefore $E$ has rank $3$. In particular $O^{p'}(\Out_\F(E)) \cong \SL_2(p)$ by Theorem \ref{typeIIauto}.
Let $\alpha \in \Aut_\F(S)$ be such that $E\alpha \neq E$. Then $S=EE\alpha$ and since $E$ is abelian we deduce that $E\alpha$ is abelian and $\Z(S)=E \cap E\alpha$. Thus $[S \colon \Z(S)]=p^2$.
By Lemma \ref{same.rank} the group $S/\Phi(E)$ has exponent $p$, $\Phi(S)=\Phi(E)[S,S]$  and $S$ has rank $3$. By Lemma \ref{char.frat} we also have that the group $\Phi(E)$ is $\F$-characteristic in $S$.
Let $C\leq E$ be the preimage in $E$ of $\C_{E/\Phi(E)}(O^{p'}(\Aut_\F(E)))$.
Then by Lemma \ref{char.C} we have $S/C \cong p^{1+2}_+$, $C < \Z(S)$ and $C=T$.
In particular $C$ is $\F$-characteristic in $S$.

Let $P$ be an $\F$-essential subgroup of $S$.
Then $\Z(S) < P < S$. So $[S \colon P] =[P \colon \Z(S)]=p$ and $P/\Phi(E)$ is elementary abelian (because $S/\Phi(E)$ has exponent $p$). Since $S$ has sectional rank $3$ we deduce that $\Phi(E)=\Phi(P)$. Thus $P$ has rank $3$ and $\Phi(E)$ is $\F$-characteristic in $P$.  Since $S$ has rank $3$, by Lemma  \ref{norm.rank3.SL2p} we deduce that $O^{p'}(\Out_\F(P)) \cong \SL_2(p)$. In particular if $H$ is the preimage in $P$ of $\C_{P/\Phi(E)}(O^{p'}(\Aut_\F(P)))$, then $[H \colon \Phi(E)]=p$. So $H/\Phi(E)$ is a maximal subgroup of $\Z(S)/\Phi(E)$. Let $\mu \in \Aut_\F(P)$ be the morphism that inverts $P/H$. Then $\mu$ centralizes $H/\Phi(E)$ by Lemma \ref{prop.T.II} and so it does not act as a scalar on $\Z(S)/\Phi(E)$. However $\mu$ is the restriction to $P$ of an $\F$-automorphism of $S$ (because $\mu \in \N_{\Aut_\F(P)}(\Aut_S(P))$,  $P$ is receptive and $S=\N_S(P)$) and so it normalizes $C/\Phi(E), H/\Phi(E)$ and $\Phi(S)/\Phi(E)$. Since $P$ is $\F$-essential, by Lemma \ref{char.series} we have $H\nleq \Phi(S)$. Hence $H=C$. In particular $C$ is $\F$-characteristic in $P$ (indeed if $P$ is not $\F$-characteristic in $S$ then $C=\coreF(P)$).

We proved that the group $C$ is $\F$-characteristic in $S$ and in every $\F$-essential subgroup of $S$. Hence by \cite[Proposition I.4.5]{AKO} we have $C\norm \F$.
Since $E$ has rank $3$ we also have $|C|\geq p$, and so $1 \neq C \leq O_p(\F)$.
\end{proof}

\begin{lemma}\label{same.frattini}
Suppose $E$ is normal in $S$ and has rank $3$. Let $P\leq S$ be an $\F$-characteristic $\F$-essential subgroup of $S$. Then $\Phi(P)=\Phi(E)$.
\end{lemma}

\begin{proof}
By Theorem \ref{typeIIauto} we have $O^{p'}(\Out_\F(E))\cong \SL_2(p)$.
Let $C\leq E$ be the preimage in $E$ of $\C_{E/\Phi(E)}(O^{p'}(\Aut_\F(E)))$. Then by Lemma \ref{char.C} we have $S/C \cong p^{1+2}_+$, $C/\Phi(E) < \Z(S/\Phi(E)$ and $C= T$. Let $\tau$ be the automorphism of $E$ that inverts $E/C$. Then $\tau$ centralizes $C$ by Lemma  \ref{prop.T.II} and since $E$ is receptive and normal in $S$, $\tau$ is the restriction of an $\F$-automorphism $\ov{\tau}$ of $S$. In particular $\ov{\tau}$ normalizes $P$ and centralizes the quotient $P/(E\cap P)$.

\textbf{Case 1:} suppose $C\leq P$.
Let $x\in P \backslash E$ and $y \in (E\cap P) \backslash C$. Then $x\ov{\tau} = xc$ for some $c\in C$ and $y\ov{\tau} =y^{-1}u$ for some $u\in \Phi(E)$. Hence using properties of commutators (\cite[Theorem 2.2.1, Lemma 2.2.2]{GOR})  we get
\[ [x,y]\ov{\tau} = [xc,y^{-1}u] \equiv [x,y]^{-1} \mod \Phi(E).\]
Since $\ov{\tau}$ centralizes $C/\Phi(E)$ and $[x,y] \in C$, we deduce that $[x,y] \equiv 1 \mod \Phi(E)$ and so the group $P/\Phi(E)$ is abelian.
Since $S$ has sectional rank $3$ and the group $S/\Phi(E)$ has exponent $p$ by Lemma \ref{same.rank}, we conclude that  $\Phi(E)=\Phi(P)$.

\textbf{Case 2:} suppose $C\nleq P$. Then $E/\Phi(E) \cong C/\Phi(E) \times (E\cap P)/\Phi(E)$ and so $E\cap P$ is $\F$-characteristic in $E$ and $(E\cap P)/\Phi(E)$ is a natural $\SL_2(p)$-module for $O^{p'}(\Out_\F(E))$.
Suppose for a contradiction that $\Phi(E) \neq \Phi(P)$. By Lemma \ref{same.rank} the group $S$ has rank $3$. Since $P$ is $\F$-essential, by Lemma \ref{strict.frattini} we have
\[ \Phi(P) < [S,S]\Phi(P) \leq \Phi(S).\]
 Thus $S/\Phi(P)$ is non-abelian, $P$ has rank $3$ and $\Phi(S)=\Phi(E)\Phi(P)$.
Since $\ov{\tau}$ centralizes $C$, it centralizes $\Phi(S)/\Phi(P)$.
Let $x\in E \backslash P$ and $y\in (E\cap P) \backslash \Phi(S)$. Then $x\ov{\tau} = xu$ for some $u\in \Phi(S)$ and $y\ov{\tau} = y^{-1}v$ for some $v\in \Phi(P)$.
 Therefore
\[ [x,y]\ov{\tau} = [xu, y^{-1}v] \equiv [x,y]^{-1} \mod \Phi(P).\]
Since $[x,y] \in [S,S] \leq \Phi(S)$, we deduce that $[x,y] \in \Phi(P)$ and so the group $E/\Phi(P)$ is abelian. In particular $(E\cap P)/\Phi(P) \leq \Z(S/\Phi(P))$.

Since $S/\Phi(P)$ is non-abelian, we get $(E\cap P)  \Phi(P) = \Z(S/\Phi(P))$ and since $P$ is $\F$-characteristic in $S$, we deduce that $E\cap P$ is $\F$-characteristic in $S$. Thus $E\cap P \leq \coreF(E)=C$, a contradiction.
\end{proof}

\begin{proof}[\textbf{Proof of Theorem \ref{normal}}]
Suppose that $S$ is not isomorphic to a Sylow $p$-subgroup of the group $\Sp_4(p)$.
Then by Theorem \ref{Sp4} there is at most one $\F$-essential subgroup of $S$ that is $\F$-characteristic in $S$.
If there exists an $\F$-essential subgroup $E$ of $S$ having rank $2$, then $E$ is an $\F$-pearl by Theorem \ref{rank2pearl}, and by \cite[Theorem B]{pearls} we conclude that $E$ is not normal in $S$, as wanted.

Suppose that all the $\F$-essential subgroups of $S$ have rank $3$.
The group $Z_S=\Omega_1(\Z(S))$ is $\F$-characteristic in $S$ and contained in every $\F$-essential subgroup of $S$. Since $O_p(\F)=1$, there exists an $\F$-essential subgroup $E$ of $S$ that moves $Z_S$.
If $E$ is not $\F$-characteristic in $S$, then $E$ is abelian by Theorem \ref{E.moving.center.is.abelian} and $E$ is not normal in $S$ by Lemma \ref{OpF}, so we are done. Suppose $E$ is $\F$-characteristic $S$. If $E$ is elementary abelian, then the assumption on the sectional rank of $S$ implies $|E|=p^3$. Note that $E$ is normal in $S$, so $[S\colon E]=p$ by Theorem \ref{class.3} and we deduce that $|S|=p^4$.  Since $E\nleq O_p(\F)=1$, there exists an $\F$-essential subgroup $P$ of $S$ that is distinct from $E$. In particular $P$ is not $\F$-characteristic in $S$ and has rank $3$. Thus $P$ is elementary abelian of order $p^3$ and it is normal in $S$, contradicting  Lemma \ref{OpF}. So $\Phi(E)\neq 1$. Note that  by Lemma \ref{same.frattini} we have $\Phi(E)=\Phi(Q)$ for every $\F$-essential subgroup of $S$ that is normal in $S$. Since $O_p(\F)=1$ we conclude that there exists an $\F$-essential subgroup of $S$ that is not normal in $S$. 
\end{proof}

\section{Proof of  Theorem \ref{main}}

Throughout this section, we assume the following hypothesis.

\begin{hp}
Let $p$ be an odd prime, let $S$ be a $p$-group of sectional rank $3$, let $\F$ be a saturated fusion  system on $S$ and let $E$ be an $\F$-essential subgroup of $S$ of rank $3$ not $\F$-characteristic in $S$. Set $T=\coreF(E)<E$, $\N^1=\N_S(E)$ and $\N^2=\N_S(\N^1)$.
\end{hp}

\begin{remark}
Recall that
\begin{itemize}
\item $O^{p'}(\Out_\F(E)) \cong \SL_2(p)$  by Theorem \ref{typeIIauto};
\item every morphism $\varphi$ in $\N_{\Aut_\F(E)}(\Aut_S(E))$ is the restriction of an automorphism of $\N^1$ (since $E$ is receptive) and if $\N^1 < S$ then $[\N^2 \colon \N^1]=p$ and $\varphi$ is the restriction of an automorphism of $\N^2$ by Lemma \ref{lift.to.N2};
\item $\Phi(E)\leq T$ and  $p^2 \leq [E\colon T]  \leq p^3$ by Theorem \ref{characterization.E}(3);
\item $\N^1$ has rank $3$ and $\N^1/\Phi(E)$ has exponent $p$ by Lemma \ref{same.rank}; and
\item $\Phi(E)$ is $\F$-characteristic in $\N^1$ by Lemma \ref{char.frat} (and so $\Phi(E) \norm \N^2$).
\end{itemize}
\end{remark}

\begin{lemma}\label{Tincenter}
If $[E \colon T]=p^2$ and $\N^1 < S$ then $T\leq \Z(\N^2)$ and $T$ is $\F$-characteristic in $\N^2$.
\end{lemma}

\begin{proof}
For $i\geq 1$ let $Z_i\leq \N^2$ be the preimage in $\N^2$ of $\Z_i(\N^2/T)$.
The group $\N^2/T$ has maximal nilpotency class (since $E/T \cong \C_p \times \C_p$ is self-centralizing in $\N^2/T$) and so $Z_3=\N^2$, $\N^2/Z_2\cong \C_p \times \C_p$ and  $[Z_2 \colon Z_1]=[Z_1 \colon T] = p$.
Also, $Z_1< E$ and since $[E,Z_2] \leq Z_1\leq E$ we get $Z_2 \leq \N^1$.
By Lemma \ref{lift.to.N2} and the fact that $O^{p'}(\Out_\F(E))\cong \SL_2(p)$ (Theorem \ref{typeIIauto}), there exists a morphism $\tau\in \Aut_\F(\N^2)$ that normalizes $E$ and inverts $E/T$. Note that $\tau$ normalizes $Z_1$ and $Z_2$ and by Lemma \ref{prop.T.II} it centralizes $T$.
 Using properties of commutators (\cite[Theorem 2.2.1, Lemma 2.2.2]{GOR}), we deduce that $\tau$ centralizes the quotient $Z_2/Z_1$ and inverts $\N^2/\N^1$. Let $C\leq \N^2$ be the preimage in $\N^2$ of $\C_{\N^2/T}(Z_2/T)$. Then $[\N^2 \colon C]=p$ and $\N^2=\N^1C$.
Let $c\in C$ be such that $C=\langle c \rangle Z_2$.
Note that $c\tau = c^{-1}z$ for some $z\in Z_2$ and for every $t\in T$ we have $[c,t] \in T$ and $[z,t]=1$.  So we get
\[ [c,t] = [c,t]\tau = [c\tau, t\tau] = [c^{-1}z, t] =[c^{-1},t].\]

Therefore $[c,t] = 1$. Since this is true for every $t\in T$, we conclude that $T\leq \Z(C)$. By Theorem \ref{characterization.E}  we have $T\leq \Z(\N^1)$, so $T\leq \Z(\N^1C)=\Z(\N^2)$. 

If $T=1$ or $T=\Z(\N^2)$ then $T$ is characteristic in $\N^2$. Suppose $1 \neq T< \Z(\N^2)$. Then $[E \colon \Z(\N^2)]=p$ and so $Z_1 = \Z(\N^2)$, $Z_2=\Z_2(\N^2)$ and $E$ is abelian.
Suppose for a contradiction that there exists $\alpha \in \Aut_\F(\N^2)$ such that $T\alpha \neq T$. Then $\N^1\alpha \neq \N^1$, $\N^1 \cap \N^1\alpha =\Z_2(\N^2)$ and $TT\alpha =\Z(\N^2)$.
The morphism $\tau$ acts as a scalar on $\N^2/\Z_2(\N^2)$.
Hence $\tau$ normalizes $\N^1\alpha$. Note that $\Aut_\F(\N^1\alpha) = \alpha^{-1}\Aut_\F(\N^1)\alpha$, so $T\alpha$ is $\F$-characteristic in $\N^1\alpha$. Thus $\tau$ normalizes $T\alpha$ and $T\cap T\alpha$.
Let $x\in \N^1\alpha \backslash \Z_2(\N^2)$ and $y \in \Z_2(\N^2) \backslash \Z(\N^2)$. Then $x\tau = x^{-1}z_1$ and $y\tau = yz_2$ for some $z_1,z_2 \in \Z(\N^2)$ and
\[ [x,y]\tau = [x^{-1}z_1, yz_2] = [x,y]^{-1}. \]
In particular, since $\tau$ centralizes $T/T\cap T\alpha$, we conclude that $[x,y] \in T\alpha$ and so $\N^1\alpha/T\alpha$ is abelian. Thus $\N^1/T \cong \N^1\alpha/T\alpha$ is abelian, contradicting  Lemma \ref{non.abelian.quotient}.
Therefore $T=T\alpha$ for every $\alpha \in \Aut_\F(\N^2)$ and so $T$ is $\F$-characteristic in $\N^2$.

\end{proof}

\begin{lemma}\label{p2norm}
If $p\geq 5$ and $[E \colon T] =p^2$ then $E\norm S$.
\end{lemma}

\begin{proof} Aiming for a contradiction, suppose that $\N^1 < S$ and set $\Phi_E=\Phi(E)$. By Lemma \ref{char.frat} we have $\Phi_E \norm \N^2$ and we can consider the group $\N^2/\Phi_E$.

Note that $E/\Phi_E$ is a soft subgroup of $\N^2/\Phi_E$. In particular by \cite[Lemma 1 and Theorem 2]{Het2},
if we denote by $Z_1$ the preimage in $\N^2$ of $\Z(\N^1/\Phi_E)$ and we set $Z_2=[\N^2,\N^2]Z_1$, then  $Z_1  < E$, $Z_2 < \N^1$ and $\N^1/Z_1 \cong \N^2/Z_2 \cong \C_p \times \C_p$. 
Let $C\leq \N^2$ be the preimage in $\N^2$ of $\C_{\N^2/T}(Z_2)$. Then $[\N^2 \colon C] = [C \colon Z_2] =p$ and $\N^2=\N^1C$. By Lemma \ref{same.rank} the group $\N^1/\Phi_E$ has exponent $p$, $\N^1$ has rank $3$ and $[\N^1,\N^1]\Phi_E=\Phi(\N^1)$.
In particular the quotient $Z_2/\Phi_E$ is elementary abelian of order $p^3$ and both $T/\Phi_E$ and $\Phi(\N^1)/\Phi_E$ are normal subgroups of $\N^2/\Phi_E$ of order $p$. Thus $T\Phi(\N^1) / \Phi_E \leq \Z(\N^2/\Phi_E) \leq Z_1/\Phi_E$ and we deduce that $Z_1$ is the preimage in $\N^2$ of $\Z(\N^2/\Phi_E)$. So $Z_1/\Phi_E$ is in the center of $C/\Phi_E$.

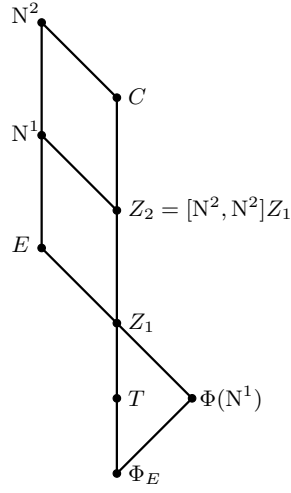
\begin{figure}[H]
\centering
\begin{tikzpicture}[x=1.00mm, y=1.00mm, inner xsep=0pt, inner ysep=0pt, outer xsep=0pt, outer ysep=0pt]
\definecolor{L}{rgb}{0,0,0}
\path[line width=0.30mm, draw=L] (-10,30) -- (-10.0,-0.);
\path[line width=0.30mm, draw=L] (-10,15) -- (-0.0,5);
\path[line width=0.30mm, draw=L] (-10,30) -- (-0.0,20);
\path[line width=0.30mm, draw=L] (-10,-0.) -- (10.,-20);
\path[line width=0.30mm, draw=L] (10,-20.) -- (0.,-30);
\definecolor{F}{rgb}{0,0,0}
\path[line width=0.15mm, draw=L, fill=F] (0.0,-20.) circle (0.50mm);
\path[line width=0.15mm, draw=L, fill=F] (-0.,5.0) circle (0.50mm);
\path[line width=0.15mm, draw=L, fill=F] (-10.0,-0.0) circle (0.50mm);
\path[line width=0.15mm, draw=L, fill=F] (-0.0,-10.) circle (0.50mm);
\path[line width=0.15mm, draw=L, fill=F] (-10,30) circle (0.50mm);
\path[line width=0.15mm, draw=L, fill=F] (-0.0,20.0) circle (0.50mm);
\path[line width=0.15mm, draw=L, fill=F] (-10.0,15) circle (0.50mm);
\path[line width=0.15mm, draw=L, fill=F] (10.0,-20) circle (0.50mm);
\draw(-14.,30.00) node[anchor=base west]{\fontsize{8.54}{10.24}\selectfont $\N^2$};
\draw(1.50,-10.78) node[anchor=base west]{\fontsize{8.54}{10.24}\selectfont $Z_1$};
\draw(-14.,-0.51) node[anchor=base west]{\fontsize{8.54}{10.24}\selectfont $E$};
\draw(-14.,14.27) node[anchor=base west]{\fontsize{8.54}{10.24}\selectfont $\N^1$};
\draw(1.50,4.22) node[anchor=base west]{\fontsize{8.54}{10.24}\selectfont $Z_2=[\N^2,\N^2]Z_1$};
\draw(1.50,19.13) node[anchor=base west]{\fontsize{8.54}{10.24}\selectfont $C$};
\draw(1.50,-21.) node[anchor=base west]{\fontsize{8.54}{10.24}\selectfont $T$};
\draw(11,-21.) node[anchor=base west]{\fontsize{8.54}{10.24}\selectfont $\Phi(\N^1)$};
\path[line width=0.30mm, draw=L] (-0.0,20.) -- (0.00,-30.0);
\draw(1.50,-31.) node[anchor=base west]{\fontsize{8.54}{10.24}\selectfont $\Phi_E$};
\path[line width=0.15mm, draw=L, fill=F] (-0.,-30) circle (0.50mm);
\end{tikzpicture}%
\end{figure}

Let $\varphi\in \Aut_\F(\N^2)$ be the morphism that acts on $E/T$ as $\begin{pmatrix} \lambda^{-1} & 0 \\ 0 & \lambda \end{pmatrix}$  with respect to the basis $\{ eT, zT\}$, for some $e \in E \backslash Z_1$, $z\in Z_1$ and $\lambda\in \GF(p)$ of order $p-1$. Such a morphism exists by Lemma \ref{lift.to.N2} and the fact that $O^{p'}(\Out_\F(E)) \cong \SL_2(p)$ (Theorem \ref{typeIIauto}).
Note that $\varphi$ centralizes $T$ by Lemma \ref{prop.T.II}. Let $x\in Z_2 \backslash Z_1$. Then $[e,x] \in [E, Z_2] \leq Z_1\backslash T$ and $x\varphi = x^au$ for some $a\in \GF(p)$ and $u\in Z_1$. Thus
\[ [e,x]^\lambda \equiv [e,x]\varphi = [e\varphi, x\varphi] = [e^{\lambda^{-1}}, x^au] \equiv [e,x]^{\lambda^{-1}a} \mod T.\]
Hence $a \equiv \lambda^2 \mod p$. In other words, the morphism $\varphi$ acts as $\lambda^2$ on $Z_2/Z_1$.
Since $[E,C] \leq Z_2 \backslash Z_1$, the same method shows that $\varphi$ acts as $\lambda^3$ on $C/Z_2$. 
Let $c\in C\backslash Z_2$.
Then $c\varphi = c^{\lambda^3}v$ for some $v\in Z_2$. Also note that $x\varphi = x^{\lambda^2}u$ for some $u\in \Phi_E$ (because $\Z_2/\Phi_E$ is elementary abelian). Since $\varphi$ centralizes $T$ and $[c,x] \in [C, Z_2] \leq  T\leq \Z(\N^2)$ by definition of $C$ and Lemma \ref{Tincenter}, we get
\[ [c,x] = [c,x]\varphi = [c^{\lambda^3}v, x^{\lambda^2}u] \equiv [c,x]^{\lambda^5} \mod \Phi_E.\]

Since $\lambda^5 \neq 1 \mod  p$, we deduce that $[c,x] \in \Phi_E$.
Therefore the group $C/\Phi$ is abelian of order $p^4$.
Moreover we have
\[ (c^p)\varphi = (c\varphi)^p= (c^{\lambda^3}v)^p = (c^p)^{\lambda^3}v^p[c^{\lambda^3},v]^{\frac{p(p-1)}{2}} \equiv (c^p)^{\lambda^3} \mod \Phi_E.\]

Since $p\geq 5$ and $\lambda\in \GF(p)$ has order $p-1$, we deduce that $\lambda^3 \not\equiv \lambda^2 \mod p$, $\lambda^3 \not\equiv \lambda \mod p$ and $\lambda^3 \not\equiv 1 \mod p$. Thus the only option is $c^p \equiv 1 \mod \Phi_E$. Hence  $C^p\leq \Phi_E$ and the group $C/\Phi_E$ is elementary abelian of order $p^4$, contradicting the fact that $S$ has sectional rank $3$.
Therefore the $\F$-essential subgroup $E$ is normal in $S$.
\end{proof}

\begin{lemma}\label{p3norm}
If $p\geq 5$ and $[E \colon T] =p^3$ then $E\norm S$.
\end{lemma}

\begin{remark}
By Theorem \ref{characterization.E}(3) the assumption $[E \colon T]=p^3$ implies that $T=\Phi(E)$ and $E$ has rank $3$.
\end{remark}

\begin{proof}
Aiming for a contradiction, suppose that $\N^1<S$. Then $[\N^2 \colon \N^1] =p$ by Lemma \ref{lift.to.N2}.
Also, $T\norm \N^2$ because $T$ is $\F$-characteristic in $\N^1$ and we can consider the group $\N^2/T$. To simplify notation we assume $T=1$.

By Theorem \ref{characterization.E}(4) and Lemma \ref{centerN2} we have $|\Z(\N^1)|=p^2$ and $|\Z(\N^2)|=p$.
Recall that $O^{p'}(\Out_\F(E))\cong \SL_2(p)$ by Theorem \ref{typeIIauto}. 
Set $C=\C_E(O^{p'}(\Aut_\F(E)))$. Then $|C|=p$ and $C$ is $\F$-characteristic in $E$, so $C \leq \Z(\N^1)$. By maximality of $T$, the group $C$ is not $\F$-characteristic in $\N^1$.
By Lemma \ref{lift.to.N2} we have $\Aut_\F(\N^1)=\Aut_S(\N^1)\N_{\Aut_\F(\N^1)}(E)$. Since $\Aut_S(\N^1) \cong \N^2/\Z(\N^1)$ and $C$ is $\F$-characteristic in $E$, we deduce that $C$ is not normal in $\N^2$. In particular $C\nleq \Z(\N^2)$.

By Lemma \ref{lift.to.N2}, there exists a morphism $\varphi \in \Aut_\F(\N^2)$ that normalizes $E$ and acts on $E/C$ as
\[ \begin{pmatrix} \lambda^{-1} & 0 \\ 0 & \lambda \end{pmatrix}, \text{ for some } \lambda\in \GF(p) \text{ of order } p-1,\]
with respect to the basis $\{xC, zC\}$, where $x\in E\backslash \Z(\N^1)$ and $z\in \Z(\N^1)\backslash C$.

Since $E$ is abelian and $[\N^1 \colon E] =p$, the group $E$ is a soft subgroup of $\N^2$. In particular if we set $H=\Z(\N^1)[\N^2,\N^2]$, then by \cite[Theorem 2]{Het2} we have that $H< \N^1$,  $\N^2/H\cong \C_p\times \C_p$ and $H$ is normalized by $\varphi$.
Hence by Maschke's Theorem (\cite[Theorem 3.3.2]{GOR})  there exists a maximal subgroup $M$ of $\N^2$ containing $H$ and distinct from $\N^1$ that is normalized by $\varphi$.
Since the action of $\varphi$ on $\Z(\N^1)$ is not scalar, we deduce that $C$ and $\Z(\N^2)$ are the only maximal subgroups of $\Z(\N^1)$ normalized by $\varphi$. Note that $\N^1$ has rank $3$ by Lemma \ref{same.rank}, so $|\Phi(\N^1)|=p$. Also, $\Phi(\N^1)\leq \Z(\N^1)$ is normalized by $\varphi$ and normal in $\N^2$. Hence $\Phi(\N^1)=\Z(\N^2)$.
In particular $[E,H] = [\N^1, \N^1] =\Z(\N^2)$.

Let  $h\in H\backslash \Z(\N^1)$. Then $h\varphi = h^au$ for some $u\in \Z(\N^1)$ and $a\in \GF(p)$ and
\[ [x,h]^{\lambda} = [x,h]\varphi = [x^{\lambda^{-1}}, h^au] = [x,h]^{\lambda^{-1}a}.\]
Hence $a \equiv \lambda^2 \mod p$.

Note that $H=\Z(\N^1)[\N^1,M]$ and $H/\Z(\N^1) = \Z(\N^2/\Z(\N^1))$. Let $y\in \N^1\backslash H$ and $g\in M \backslash H$. Then $y\varphi = y^{\lambda^{-1}}v$ for some $v\in \Z(\N^1)$ and $g\varphi  =g^bk$ for some $k\in H$ and $b\in \GF(p)$. Since $[y,g] \in H \backslash\Z(\N^1)$, we have
\[ [y,g]^{\lambda^2} \equiv [y,g]\varphi= [y^{\lambda^{-1}}u, g^bk] \equiv  [y,g]^{\lambda^{-1}b} \mod \Z(\N^1).\]
Hence $b \equiv \lambda^3 \mod p$.

Therefore $\varphi$ acts as $\lambda^2$ on $H/\Z(\N^1)$ and as $\lambda^3$ on $M/H$.

Note that
\[ [h,g] = [h,g]\varphi = [h^{\lambda^2}, g^{\lambda^3}k] \equiv  [h,g]^{\lambda^5} \mod \Z(\N^2).\]
Since $\lambda^5 \not\equiv 1 \mod p$, we deduce that $[g,h]\in \Z(\N^2)$ and the group $M/\Z(\N^2)$ is abelian. In particular $H/\Z(\N^2)=\Z(\N^2/\Z(\N^2))$ and so $H=\Z_2(\N^2)$.
Also, the group $H$ is elementary abelian (since $\N^1$ has exponent $p$ by Theorem \ref{characterization.E}(2) and $[H \colon \Z(\N^1)]=p$).

Let $c\in C$.
Then $c\varphi = c$ and we have
\[ [c,g]\varphi = [c, g^{\lambda^3}k] = [c,g]^{\lambda^{3}}.\]
Note that $[c,g] \in [\Z(\N^1),M]=\Z(\N^2)$ (since $[\Z(\N^1),M]$ is a proper subgroup of $\Z(\N^1)$ normalized by $\varphi$ and $C$ is not normal in $M$).
The assumption $p\geq 5$ implies $\lambda^3 \not\equiv \lambda \mod p$. So $[c,g]=1$ and $C\leq \Z(M)$, contradicting the fact that $C$ is not contained in $\Z(\N^2)$.

Therefore the $\F$-essential subgroup $E$ of $S$ is normal in $S$.
\end{proof}

\begin{theorem}\label{pgeq5}
Suppose $p\geq 5$, $S$ is a $p$-group of sectional rank $3$ and $\F$ is a saturated fusion system on $S$. Then every $\F$-essential subgroup of $S$ of rank $3$ is normal in $S$.
\end{theorem}

\begin{proof}
Let $E\leq S$ be an $\F$-essential subgroup of $S$ of rank $3$. If $E$ is $\F$-characteristic in $S$ then $E\norm S$. If $E$ is not $\F$-characteristic in $S$ then $\coreF(E)<E$ and by Theorem \ref{characterization.E}(3) we have $p^2 \leq [E \colon \coreF(E)] \leq p^3$. Thus $E$ is normal in $S$ by Lemmas \ref{p2norm} and \ref{p3norm}.
\end{proof}

\begin{proof}[\textbf{Proof of Theorem \ref{main}}]
By assumption $p\geq 5$ and $O_p(\F)=1$.
If $S$ is isomorphic to a Sylow $p$-subgroup of the group $\Sp_4(p)$ then the subgroups of $S$ that are candidates for $\F$-essential subgroups are the $\F$-pearls and the unique elementary abelian maximal subgroup $A\cong \C_p \times \C_p \times \C_p$. The assumption $O_p(\F) = 1$ implies that $A$ is not the only $\F$-essential subgroup of $S$, so $\F$ contains an $\F$-pearl.
Suppose $S$ is not isomorphic to a Sylow $p$-subgroup of the group $\Sp_4(p)$. Then by Theorem \ref{normal} there exists an $\F$-essential subgroup $E$ of $S$ that is not normal in $S$. Thus Theorem \ref{pgeq5} implies that $E$ has rank $2$ and by Theorem \ref{rank2pearl} we conclude that $E$ is an $\F$-pearl. Therefore in any case the fusion system $\F$ contains an $\F$-pearl.
The characterization of $S$ and $\F$ is then a direct consequence of \cite[Theorem B]{pearls}.
\end{proof}
\bibliographystyle{alpha}
\bibliography{bibrank}
\end{document}